\documentclass{amsart}
\usepackage{amsmath}
\usepackage{amsthm}
\usepackage{amsfonts,mathrsfs}
\usepackage{amssymb}
\usepackage{graphicx}
\usepackage{enumitem}
\usepackage{color}
\usepackage{verbatim}
\usepackage{hyperref}

\theoremstyle{plain}
\newtheorem{theorem}{Theorem}[section]
\newtheorem{lemma}[theorem]{Lemma}
\newtheorem{example}[theorem]{Example}

\newtheorem{proposition}[theorem]{Proposition}

\theoremstyle{definition}

\newtheorem{question}[theorem]{Question}

\newtheoremstyle{TheoremNum}
	{\topsep}{\topsep}              %%% space between body and thm
  {\itshape}                      %%% Thm body font
  {}                              %%% Indent amount (empty = no indent)
  {\bfseries}                     %%% Thm head font
  {.}                             %%% Punctuation after thm head
  { }                             %%% Space after thm head
  {\thmname{#1}\thmnote{ \bfseries #3}}%%% Thm head spec

\newcommand{\Z}{\mathbb Z}

\newcommand{\cC}{\mathcal C}

\newcommand{\RN}[1]{%
  \textup{\uppercase\expandafter{\romannumeral#1}}%
}

 \def\zzj#1 {\fbox {\footnote {\ }}\ \footnotetext { From Yue: {\color{red}#1}}}

 \def\zy#1 {\fbox {\footnote {\ }}\ \footnotetext { From Zijian: {\color{blue}#1}}}

 \newcommand{\myue}[1]{{{\color{red}#1}}}

\begin{document}
	\title[APLL codes of packing radius $2$ only exist for small dimensions]{Almost perfect linear Lee codes of packing radius $2$ only exist for small dimensions}
	\author[Z.\ Zhou]{Zijian Zhou\textsuperscript{\,1}}
	\author[Y.\ Zhou]{Yue Zhou\textsuperscript{\,1, $\dagger$}}
	\address{\textsuperscript{1}Department of Mathematics, National University of Defense Technology, 410073 Changsha, China}
	\address{\textsuperscript{$\dagger$}Corresponding author}
	\email{yue.zhou.ovgu@gmail.com}
	\date{\today}
	\begin{abstract}
		It is conjectured by Golomb and Welch around half a century ago that there is no perfect Lee codes $C$ of packing radius $r$ in $\mathbb{Z}^{n}$ for  $r\geq2$ and $n\geq 3$. Recently, Leung and the second author proved this conjecture for linear Lee codes with $r=2$. A natural question is whether it is possible to classify the second best, i.e., almost perfect linear Lee codes of packing radius $2$. We show that if such codes exist in $\mathbb{Z}^n$, then $n$ must be $1,2,
		11, 29, 47, 56, 67, 79, 104, 121, 134$ or $191$.
	\end{abstract}
	\maketitle
	
\section{Introduction}
Let $\mathbb{V}$ be a metric space. A code $\cC$ of $\mathbb{V}$ is just a subset of points in $\mathbb{V}$. The code $\cC$ is a \emph{perfect code} with radius $r$, if for any point $x$ in $\mathbb{V}$, there is exactly one point in $\cC$ with distance at most $r$ from $x$. Usually, perfect codes yields beautiful mathematical structures and they also have very important applications in the associated metric spaces.

In this paper, we are interested in perfect and almost perfect codes with respect to the Lee metric. For other interesting topics on perfect codes, we refer to the monograph \cite{etzion_perfect_2022} by Etzion. Let $\Z$ denote the ring of integers. For $x=(x_1,\cdots, x_n)$ and $y=(y_1,\cdots, y_n)\in \Z^n$, the \emph{Lee metric} (also known as $\ell_1$-norm, taxicab metric, rectilinear distance or Manhattan distance) between them is defined by
\[d_L(x,y)=\sum_{i=1}^n |x_i-y_i| \text{ for }x,y\in\Z^n.\]
A Lee code $\cC$ is called \emph{linear}, if $\cC$ is a lattice in $\Z^n$. The \emph{minimum distance} $d(\cC)$ of a Lee code $\cC$ containing at least two elements is defined to be $\min\{d_L(x,y): x,y\in \cC, x\neq y\}$.

Lee codes have many practical applications such as interleaving schemes \cite{blaum_interleaving_1998}, constrained and partial-response channels \cite{roth_lee-metric_1994}, and flash memory \cite{schwartz_quasi-cross_2012}.

By definition, a Lee code $\cC$ is perfect if and only if
$$\Z^n =\dot{\bigcup}_{c\in \cC} (S(n,r)+c),$$
where 
\[S(n,r):= \left\{(x_1,\cdots,x_n)\in \Z^n: \sum_{i=1}^n |x_i|\leq r \right\}\] 
is the Lee sphere of radius $r$ centered at the origin and $S(n,r)+c := \{v+c: v\in S(n,r) \}$. In another word, the existence of a perfect Lee code is equivalent to a tiling of $\Z^n$ by Lee spheres of radius $r$. The cardinality of $S(n,r)$ has been determined in \cite{golomb_perfect_1970,stanton_note_1970} as follows
\[
|S(n,r)|=\sum_{i=0}^{\text{min}\{n,r\}}2^{i}\binom {n}{i}\binom {r}{i}.
\] 

It is worth pointing out that a perfect linear Lee code is also equivalent to an abelian Cayley graph of degree $2n$ and diameter $r$ whose number of vertices meets the so-called \emph{abelian Cayley Moore bound}; see \cite{camarero_quasi-perfect_lee_2016,zhang_nonexistence_2019}. For more results about the degree-diameter problems in graph theory, we refer to the survey \cite{miller_moore_2013}.

It has been showed by Golomb and Welch \cite{golomb_perfect_1970}  that perfect Lee codes exist for $n=1,2$ and every positive integer $r$, and for $n\geq 3$ and $r=1$. In the same paper, Golomb and Welch also proposed the conjecture that there are no more perfect Lee codes for other choices of $n$ and $r$. To the best of our knowledge, this conjecture is still far from being solved, despite many efforts and various approaches applied on it. For more history and results about this conjecture, we refer to the survey \cite{horak_50_2018} and the references therein.

In recent years, there are some important progress on the existence of perfect linear Lee code for small $r$. For $r=2$, Leung and the second author \cite{leung_lattice_2020}
proved that perfect linear Lee codes only exist for $n=1,2$. For $r\geq 2$, Zhang and Ge \cite{zhang_perfect_lp_2007}, and Qureshi \cite{qureshi_nonexistence_ZGcondition_2020} have obtained many nonexistence result by extending the symmetric polynomial method introduced by Kim \cite{kim_2017_nonexistence} for $r=2$. It is worth pointing out that by \cite{szegedy_algorithms_1998} the linearity assumption of a nonexistence result for perfect linear Lee codes of radius $r$ in $\Z^n$ can be removed provided that the size of Lee sphere $S(n,r)$ is a prime number.

Obviously, a Lee code $\cC$ is perfect implies that the packing of $\Z^n$ by $S(n,r)$ with centers consisting of all the elements in $\cC$ is of packing density $1$. 
As there is no perfect linear Lee code for $r= 2$ and $n\geq 3$ by \cite{leung_lattice_2020}, one natural question is whether the second best is possible. More precisely, one may ask the following question for any $r\geq 2$.
\begin{question}\label{ques:big}
	Does there exist a lattice packing of $\Z^n$ by $S(n,r)$ with density $\frac{|S(n,r)|}{|S(n,r)|+1}$?
\end{question}
If such a lattice packing exists, we call the associated linear Lee code \emph{almost perfect}. For convenience, we abbreviate the term almost perfect linear Lee code to \emph{APLL} code throughout the rest of this paper.

The \emph{packing radius} (\emph{covering radius}, resp.) of a Lee code $C$ is defined to be the largest (smallest resp.) integer $r'$ such that for any element $w\in\Z^n$  there exists at most (at least, resp.) one codeword $c\in C$ with $d_L(w,c)\leq r'$. It is clear that a Lee code is perfect if and only if its packing radius and covering radius are the same.
For an APLL code with packing density $\frac{|S(n,r)|}{|S(n,r)|+1}$,  its 
packing radius and covering radius are $r$ and $r+1$, respectively. 

%\begin{figure}[h!]
%	\centering
%	\includegraphics[width=0.9\textwidth]{Z6_Lee_ball_nonumber}
%	\caption{An almost perfect linear Lee code in $\Z$}
%	\label{fig:Z_6}
%\end{figure}
It is easy to see that $\cC=\{x(2r+2): x\in \Z\}$ is the unique APLL code of packing radius $r$ in $\Z$ for any $r$. 

\begin{figure}[h!]
	\centering
	\includegraphics[width=0.75\textwidth]{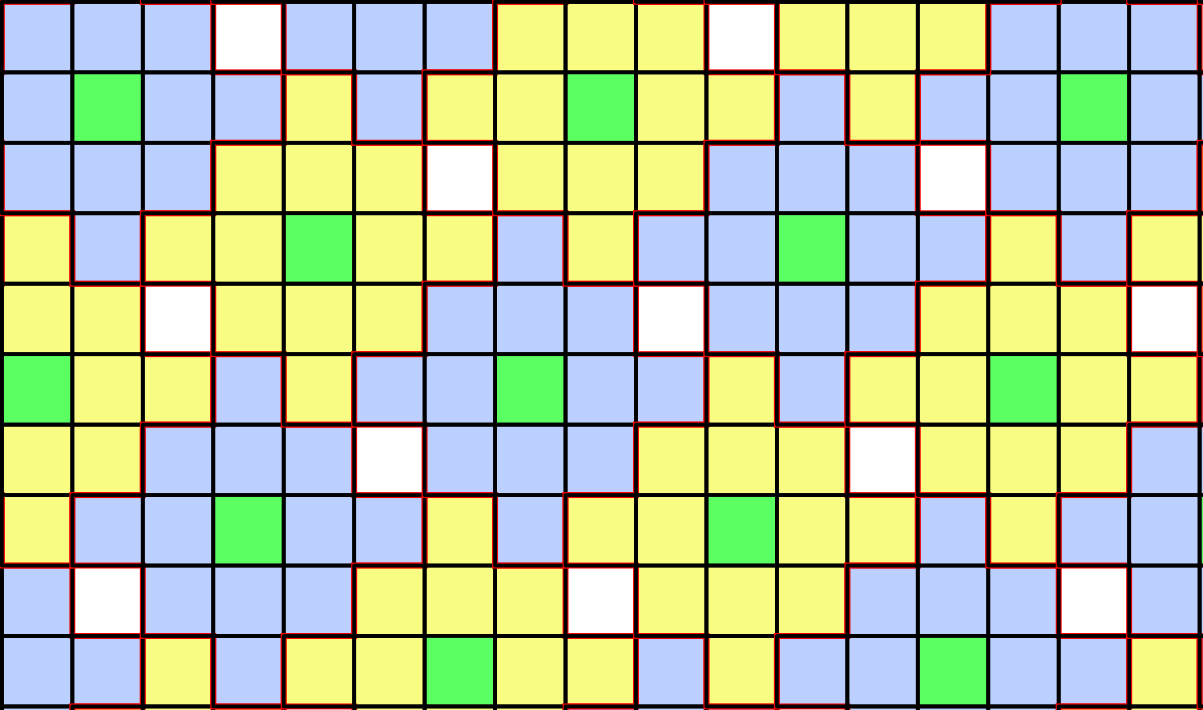}
	\caption{Packing of $\mathbb{R}^2$ by $L(2,2)$ associated with an APLL code}
	\label{fig:Z_14}
\end{figure}

Let $L(n,r)$ denote the union of $n$-cubes centered at each point of $S(n,r)$ in $\mathbb{R}^n$. It is clear that each packing of $\Z^n$ by $S(n,r)$ corresponds to a packing of $\mathbb{R}^n$ by $L(n,r)$.  Figure \ref{fig:Z_14} depicts the lattice packing of $\mathbb{R}^2$ by $L(2,2)$ associated with a lattice generated by $\{(1,4), (3,-2) \}$ in $\Z^2$. Clearly the packing density is $\frac{|S(2,2)|}{|S(2,2)|+1}=\frac{13}{14}$. Therefore this lattice is an APLL code in $\Z^2$. It is not difficult to show that the minimum distance of an APLL code in $\Z^n$ is $2r+1$ for $n>1$ and $2r+2$ for $n=1$.
%This example and the numbers labeled on the cubes will be explained later in Example \ref{ex:G_n=1,2}. 

In \cite{xu_wcc_2022}, the following result on the nonexistence of APLL codes of packing radius $2$ has been obtained.
\begin{theorem}\label{th:main_xu}
	Let $n$ be a positive integer. If $n\equiv 0,3,4\pmod{6}$, 
	then there exists no almost perfect linear Lee code of packing radius $2$.
\end{theorem}

In this paper, we consider the rest cases with $n\equiv 1,2,5 \pmod{6}$. Together with Theorem \ref{th:main_xu}, we can prove the following result.
\begin{theorem}\label{th:main}
	There is no almost perfect linear Lee code of packing radius $2$ in $\Z^n$ for any positive integer $n$ except for $n=1,2,
	11, 29, 47, 56, 67, 79, 104, 121, 134, 191$.
\end{theorem}

The rest part of this paper is organized as follows. In Section \ref{sec:pre}, we recall some necessary and sufficient conditions of the existence of APLL codes of packing radius $2$ in terms of group ring equations. In Section \ref{sec:outline}, we recall how to prove Theorem \ref{th:main_xu} in \cite{xu_wcc_2022} and provide an outline of the proof of Theorem \ref{th:main}. As this proof is very long, we separate it into several lemmas and theorems in Sections \ref{sec:X_iY_i}--\ref{sec:1mod6}, which tell us that $n$ is bounded if there exist APLL codes of packing radius $2$ in $\Z^n$ for $n\equiv 1,2,5 \pmod{6}$. In Section \ref{sec:small_n}, by two nonexistence results from \cite{he_thesis_2021} and \cite{he_nonexistence_abelian_2021}, we can exclude most of the rest small values of $n$ and finish the proof of Theorem \ref{th:main}.  In Section \ref{sec:concluding}, we conclude this paper with a conjecture and several remarks.

\section{Preliminaries}\label{sec:pre}
Let $G$ be a finite group. Let $\Z[G]$ denote the set of formal sums $\sum_{g\in G} a_g g$, where $a_g\in \Z$ and $G$ is a finite multiplicative group. The addition and the multiplication on $\Z[G]$ are defined by
$$\sum_{g\in G} a_g g +\sum_{g\in G} b_g g :=\sum_{g\in G} (a_g+b_g) g,$$
and
$$(\sum_{g\in G} a_g g )\cdot (\sum_{g\in G} b_g g) :=\sum_{g\in G} (\sum_{h\in G} a_hb_{h^{-1}g})\cdot g,$$
for $\sum_{g\in G} a_g g, \sum_{g\in G} b_g g\in\Z[G]$. Moreover,
$$\lambda \cdot(\sum_{g\in G} a_g g ):= \sum_{g\in G} (\lambda a_g) g $$
for $\lambda\in\Z$ and $\sum_{g\in G} a_g g\in \Z[G]$.

For an element $A=\sum_{g\in G} a_g g\in \Z[G]$ and $t\in \Z$, we define 
$$A^{(t)}:=\sum_{g\in G} a_g g^t.$$
A multi-subset $D$ of $G$ can be viewed as an element $\sum_{g\in D} g\in \Z[G]$. In the rest of this paper, by abuse of notation, we will use the same symbol to denote a multi-subset in $G$ and the associated element in $\Z[G]$. Moreover, $|D|$ denotes the number of distinct elements in $D$. In another word, $|D|$ is the cardinality of the \emph{support} of $D$ which is defined by $\mathrm{Supp}(D):=\{g\in G: \text{ the coefficient of $g$ in $D$ is positive}\}$.

In \cite[Section 3]{xu_wcc_2022}, Xu and the second author have proved the equivalence between the existence of an APLL code of packing radius $2$ and some group ring equations.

\begin{lemma}\label{le:group_ring_inter}
	There exists an APLL code of packing radius $2$ in $\Z^n$ if and only if there is an abelian group  $H$ of order $n^2+n+1$ containing two inverse-closed subsets $T_0$ and $T_1$ such that the identity element $e\in T_0$, and the following two group ring equations in $\Z[H]$ holds.
	\begin{eqnarray}
	\label{eq:T0T1} T_0T_1&=&H-e,\\
	\label{eq:T0^2+T1^2} T_0^2+T_1^2&=&2H-T_0^{(2)}-T_1^{(2)}+2ne.
	\end{eqnarray}
\end{lemma}

\begin{example}\label{ex:H_n=1,2}
	The following examples  of $T_0$ and $T_1$  satisfy \eqref{eq:T0T1} and \eqref{eq:T0^2+T1^2}.
	\begin{enumerate}[label=(\alph*)]
		\item For $n=1$, $H=\langle h \rangle \cong C_{3}$. Let $T_0=\{e\}$ and $T_1=\{h,h^2\}$.
		\item 	For $n=2$, $H=\langle h\rangle \cong C_7$.  Define
		\[ T_0=\{e,h,h^{6}\} \text{ and }T_1=\{h^2,h^{5}\}. \]
	\end{enumerate}
\end{example}

The following result is a collection of necessary conditions for $T_0$ and $T_1$, which is also proved in \cite{xu_wcc_2022}.
\begin{lemma}\label{le:T_0T_1_nece_condi}
	Let $H$ be an abelian group of order $n^2+n+1$. Suppose $T_0,T_1$ are subsets of $H$ satisfying $e\in T_0$, $T_i^{(-1)}=T_i$ for $i=0,1$, \eqref{eq:T0T1} and \eqref{eq:T0^2+T1^2}.  Then the following statements hold.
	\begin{enumerate}[label=(\alph*)]
		\item $e\in T_0$, $e\notin T_1$;
		\item $T_0 \cap T_1=\emptyset$ and $T_0^{(2)} \cap T_1^{(2)}=\emptyset$;
		\item $T_0\cap (T_0^{(2)}\setminus \{e\})=T_0\cap T_1^{(2)}=\emptyset$;
		\item $\{ab: a\neq b,  a,b\in T_0\}\cap T_0^{(2)} =\{e\}$; 
		\item When $n$ is odd, $|T_0|= n$ and $|T_1|=n+1$;
		\item When $n$ is even, $|T_0|= n+1$ and $|T_1|=n$;
		\item There is no common non-identity element in $T_0^2$ and $T_1^2$;
		\item $T_0\cap T_0^{(3)} = \{e\}$.
	\end{enumerate}
\end{lemma}

\section{Outline of the method}\label{sec:outline}
First, let us take a look at the sketch of the proof of the nonexistence of APLL codes of packing radius $2$ with $n\equiv 0,3,4\pmod{6}$ in \cite{xu_wcc_2022}.

By Lemma \ref{le:group_ring_inter}, we only have to show there is no $T_0$ and $T_1$ in $H$ satisfying the necessary and sufficient conditions in it.  Define $\hat{T}=T_0+T_1\in \Z[H]$, $k_0=|T_0|$ and $k_1=|T_1|$.

Multiplying $T_0$ and $T_1$ on both sides of $(\ref{eq:T0^2+T1^2})$, we get 
\begin{align}
T_0^3=(2k_0-k_1)H+2nT_0+T_1-T_0\hat{T}^{(2)}\, , \label{eq:T0^3left}\\
T_1^3=(2k_1-k_0)H+2nT_1+T_0-T_1\hat{T}^{(2)}\, .  \label{eq:T1^3left}
\end{align}

Consider the above two equations modulo $3$ 
\begin{align*}
\hat{T}^{(2)}T_0 &\equiv (2k_0-k_1)H+T_1-T_0^{(3)}+2nT_0 \pmod{3},\\
\hat{T}^{(2)}T_1 &\equiv (2k_1-k_0)H+T_0-T_1^{(3)}+2nT_1 \pmod{3}.
\end{align*}
Note that $T_i^3 \equiv T_i^{(3)} \pmod{3}$ for $i=0,1$. If $3\nmid |H|$, then $T_0^{(3)}$ and $T_1^{(3)}$ are both subsets of $H$; for $3\mid |H|$, we can also show that  there are at most 2 elements in $T_i^{(3)}$ appearing twice for $i=0,1$; see Lemma \ref{le:n=1mod3_rep}. Hence, by calculation, we can derive some strong restrictions on the coefficients of most of the elements in the right-hand side of the above two equations. For instance, when $n\equiv 1 \pmod{3}$ and $n$ is odd, the first equation becomes
\[
\hat{T}^{(2)}T_0 \equiv T_1-T_0^{(3)}+2T_0 \pmod{3}.
\]
As $T_1$, $T_0^{(3)}$ and $T_0$ are approximately of size $n$, most of the elements in $H$ appear in $\hat{T}^{(2)}T_0$ for $3k$ times, $k=0,1,\cdots$.

Let $X_i$ ($Y_i$, resp.) be the subset of elements of $H$ appearing in $\hat{T}^{(2)} T_0$ ($\hat{T}^{(2)} T_1$, resp.) exactly $i$ times for $i=0,1,\cdots, M_0$ ($M_1$, resp.), which means $X_i$'s ($Y_i$'s, resp.) form a partition of the group $H$. In particular, we define $M_0$ and $M_1$ to be the largest integers such that $X_{M_0}$ and $Y_{M_1}$ are non-empty sets. Then
\begin{align}
\label{eq:T^2T_0} \hat{T}^{(2)} T_0 &= \sum_{i=0}^{M_0} iX_i,\\
\label{eq:T^2T_1} \hat{T}^{(2)} T_1 &= \sum_{i=0}^{M_1} iY_i.
\end{align}

By \eqref{eq:T^2T_0} and \eqref{eq:T^2T_1}, we have the following three conditions on the value of $|X_i|$'s and $|Y_i|$'s:
\begin{align}
\label{eq:sum_i|X_i|}	\sum_{i=1}^{M_0}i|X_i| &=(2n+1)k_0,\\
\label{eq:sum_i|Y_i|}	\sum_{i=1}^{M_1}i|Y_i| &=(2n+1)k_1,\\
\label{eq:sum_|X_i|}	\sum_{i=0}^{M_0}|X_i|=\sum_{i=0}^{M_1}|Y_i|&= n^2+n+1.
\end{align}
Some extra conditions can also be derived:
	\begin{align}
	\label{eq:sum_|X_i|-0} \sum_{i=1}^{M_0}|X_i|=&(2n+1)k_0-2(k_0-1)k_0+\theta_0 + \sum_{s\geq 3} \frac{(s-1)(s-2)}{2}|X_s|,\\
	\label{eq:sum_|Y_i|-0}  \sum_{i=1}^{M_1}|Y_i|=&(2n+1)k_1-2(k_1-1)k_1 + \theta_1+ \sum_{s\geq 3} \frac{(s-1)(s-2)}{2}|Y_s|,
	\end{align}
	where 
	\begin{equation}\label{eq:theta_0}
		\theta_0=|(T_0^2 \setminus T_0^{(2)})\cap \hat{T}^{(4)}|,
	\end{equation}
	and 
	\begin{equation}\label{eq:theta_1}
		\theta_1 = \frac{|T_1\cap \hat{T}^{(2)}|}{2}+|(T_1^2\setminus (T_1^{(2)} \cup \{e\}))\cap \hat{T}^{(4)}|;
	\end{equation}
	see \cite[Lemma 3.1]{xu_wcc_2022}.
	
Note that $|X_i|$'s and $|Y_i|$'s are nonnegative integers. Our main approach in \cite{xu_wcc_2022} is to use the 5 equations \eqref{eq:sum_i|X_i|}	, \eqref{eq:sum_i|Y_i|}	, \eqref{eq:sum_|X_i|}, \eqref{eq:sum_|X_i|-0} and \eqref{eq:sum_|Y_i|-0} together with $\hat{T}^{(2)}T_0$ and $\hat{T}^{(2)}T_1 \pmod{3}$ to find contradictions on the value of $|X_i|$'s and $|Y_i|$'s provided that $n$ is not too small.

However, for $n\equiv 1, 2, 5\pmod{6}$, the above approach does not lead to any contradiction. Instead, we can only determine the value or the ranges of $|X_i|$'s and $|Y_i|$'s. These results are provided in Section \ref{sec:X_iY_i}.

To obtain further nonexistence results, one possible strategy is to look at $\hat{T}^{(4)}T_0$ and $\hat{T}^{(4)}T_1$ modulo $5$. To be precise, we concentrate on the necessary conditions
\begin{align}
\hat{T}^{(4)}T_0=(5k_0+(2k_0-k_1)(k_0^2-1))H+(4n^2+2n-3)T_0+2nT_1-4n\hat{T}^{(2)}T_0-\hat{T}^{(2)}T_1-T_0^5,\label{eq:T0hatT^4}\\
\hat{T}^{(4)}T_1=(5k_1+(2k_1-k_0)(k_1^2-1))H+(4n^2+2n-3)T_1+2nT_0-4n\hat{T}^{(2)}T_1-\hat{T}^{(2)}T_0-T_1^5; \label{eq:T1hatT^4}
\end{align}
which will be proved in Lemma \ref{le:T4Ti}. 

As $|H|=n^2+n+1$ is always relatively prime to $5$, $T_0^{(5)}$ and $T_1^{(5)}$ are both subsets of $H$. The above two equations modulo $5$ provide us more conditions. For instance, when $n$ satisfies $n\equiv 1 \pmod{5}$, $n\equiv 1 \pmod{6}$ and $2\nmid n$, the equation \eqref{eq:T1hatT^4} leads to
\[
\hat{T}^{(4)}T_1=4H+4T_1+3T_0+4\sum_{i=1}^{M_0}iX_i+4\sum_{i=1}^{M_1}iY_i-T_1^{(5)} \pmod{5}.
\]

Similar to \eqref{eq:T^2T_0} and \eqref{eq:T^2T_1}, we set
\begin{equation}\label{eq:Ui}
	\hat{T}^{(4)}T_0=\sum_{i=0}^{N_0}iU_i,
\end{equation}
and
\begin{equation}\label{eq:Vi}
\hat{T}^{(4)}T_1=\sum_{i=0}^{N_1}iV_i\, 
\end{equation}
where $U_i$ and $V_i$ are subsets of $H$ such that $U_i$ (resp.\ $V_i$) consists of the elements appearing in $\hat{T}^{(4)}T_0$ (resp.\ $\hat{T}^{(4)}T_1$) exactly $i$ times.

By combining the known restrictions on $|X_i|$'s and $|Y_i|$'s with the group ring equations on $\hat{T}^{(4)}T_0$ and $\hat{T}^{(4)}T_1$ modulo $5$, we can prove that $n$ must be small for $n\equiv 2,5 \pmod{6}$; see Theorems \ref{th:n=2mod3_odd} and \ref{th:n=2mod3_even}. However the computation is quite complicated, because we have to separate the proof into 5 different cases depending on the value of $n$ modulo $5$ for $n\equiv 2\pmod{6}$ and for $n\equiv 5\pmod{6}$, respectively.

For the last case with $n\equiv 1 \pmod{6}$, we know the precise value of $|X_i|$'s; see Theorem \ref{th:n=1_mod6_Xi}. The same approach for the proof with $n\equiv 2,5 \pmod{6}$ also works here if $n\equiv1,4 \pmod{5}$. For $n\equiv 2 \pmod{5}$ we need a further restriction
\[
	2|V_1|+3|V_2|+3|V_3|+2|V_4|\geq 2n^2+4n+2,
\]
which is proved in Lemma \ref{le:in-ex_V_n=1mod6}. For $n\equiv 0,3 \pmod{5}$, the nonexistence result can be derived directly from Corollary 3.1 in \cite{he_nonexistence_abelian_2021}.

The results in Sections \ref{sec:2mod3} and \ref{sec:1mod6} provide upper bounds of $n$ for which APLL codes of packing radius $2$ exists. In Section \ref{sec:small_n}, we can apply some known nonexistence results in  \cite{he_thesis_2021} and \cite{he_nonexistence_abelian_2021} to exclude most of the rest small values of $n$, which leads to the final proof of Theorem  \ref{th:main}.

\section{Restrictions on $|X_i|$'s and $|Y_i|$'s}\label{sec:X_iY_i}

When $n\equiv 1\pmod{3}$, $3$ always divides $|H|$. Hence $T_0^{(3)}$ and $T_1^{(3)}$ are not necessarily subsets of $H$, i.e., it is possible that some elements in them appear more than once. In \cite{xu_wcc_2022}, the following result has been proved.
\begin{lemma}\label{le:n=1mod3_rep}
	When $n\equiv 1 \pmod{3}$, 
	\begin{enumerate}[label=(\alph*)]
		\item there is no element in $T_j^{(3)}$ appearing more than $2$ times for $j=0,1$,
		\item $e$ appears only once in $T_0^{(3)}$, and
		\item there are $0$ or $2$ elements appearing twice in $T_0^{(3)}$, and there are at most $2$ elements appearing twice in $T_1^{(3)}$.
	\end{enumerate}
\end{lemma}

Let $\Delta_i$ be the set of elements appearing twice in $T_i^{(3)}$ for $i=1,2$. By Lemma \ref{le:n=1mod3_rep}, $|\Delta_0|, |\Delta_1|\leq 2$ and there is no element appearing more than 2 times in $T_i^{(3)}$ for $i=1,2$.

\begin{theorem}\label{th:n=1_mod6_Xi}
	For $n\equiv 1 \pmod{6}$, suppose that there exist inverse-closed subsets $T_0$ and $T_1\subseteq H$ with $e\in T_0$ and $k_0+k_1=2n+1$ satisfying \eqref{eq:T0T1} and \eqref{eq:T0^2+T1^2}. Then the following results hold.
	\begin{enumerate}[label=(\alph*)]
		\item $|X_0|=\frac{1}{3}(n-1)^2$, $|X_1|=n+2, |X_2|=2n-2$, $|X_3|=\frac{2}{3}(n-1)^2, |X_i|=0$ for $i>3$.
		\item $T_1\cap T_0^{(3)}=\emptyset$, $\theta_0=0$ and $\Delta_0\subseteq T_1$.
		\item Set $u_0=|T_1^{(3)} \cap T_0|$, $u_1=|T_1^{(3)} \cap T_1|$, $\Delta_1^0 =\Delta_1\cap T_0$, $\Delta_1^1 =\Delta_1\cap T_1$ and $\Delta_1^2 =\Delta_1\setminus (T_0\cup T_1)$. Then
		\begin{equation}\label{eq:Y_3bd,n=1mod6}
			\frac{2n^2-5n-4+\theta_1}{3}+2u_0+u_1+2|\Delta_1^1|+|\Delta_1^2|\leq |Y_3|\leq \frac{2n^2-2n-3}{3}+u_0+u_1+|\Delta_1^1|+|\Delta_1^2|,
		\end{equation}
	and
		\begin{equation}\label{eq:Y_0bd,n=1mod6}
			\frac{n^2-10n-3}{3}-u_0-u_1-|\Delta_1^1|-|\Delta_1^2|+\theta_1\leq |Y_0|\leq \frac{n^2-n+1-\theta_1}{3}.
		\end{equation}
	\end{enumerate}
\end{theorem}
\begin{proof}
	As $2\nmid n$, $k_0=n$ and $k_1=n+1$ by Lemma \ref{le:T_0T_1_nece_condi} (e). 
	
	First we prove (a). Define $\Delta_0^1 = \Delta_0\cap T_1$ and $\Delta_0^2 = \Delta_0\setminus T_1$. As $T_0\cap T_0^{(3)}=\{e\} \not\subseteq\Delta_0$,  $\Delta_0=\Delta_0^1~\dot{\cup}~ \Delta_0^2$.
	
	By comparing the elements in
	\[\hat{T}^{(2)}T_0 \equiv T_1-T_0^{(3)} + 2T_0\pmod{3},\]
	we obtain
	$$\bigcup _{i\geq 0}X_{3i+2}=\left(T_0\setminus \{e\}\right) \dot{\cup} \left(T^{(3)}_0\setminus \left(\{e\}~\dot{\cup}~ \Delta_0^2~\dot{\cup}~ (T_1\cap T_0^{(3)} )\right)\right)~\dot{\cup}~ \Delta_0^1,$$ 
	in which  we also need $e\notin T_1$ and that $e$ is disjoint from $\Delta_0$ by \cite[Lemma 3.4 (b)]{xu_wcc_2022}.
	Consequently, 
	\begin{align*}
	\sum_{i\geq 0}|X_{3i+2}|&= (n-1) + (n-|\Delta_0|-1-|\Delta_0^2|-|T_1\cap T_0^{(3)} | + |\Delta_0^1|)\\
	&=2n-2-\ell_0 -2|\Delta_0^2|,
	\end{align*}
	where $\ell_0=|T_1\cap T_0^{(3)}|$.
	By a similar analysis, we can get
	\[
	\bigcup _{i\geq 0}X_{3i+1}=
	\left(T_1\setminus T_0^{(3)}\right)~\dot{\cup}~\{e\}~\dot{\cup}~\Delta_0^2.
	\]
	To summarize, we have obtained
	\begin{equation}\label{eq:1:X_i_n=0}		
		\sum_{i\geq 0}|X_{3i+2}|=2n-\ell_0-2|\Delta_0^2|-2, 
		\qquad \sum_{i\geq 0}|X_{3i+1}| =n+2-\ell_0 +|\Delta_0^2|.
	\end{equation}
	
	Now \eqref{eq:sum_|X_i|-0} is 
	\begin{align*}
	\sum_{i=1}^{M_0}|X_i|&=(2n+1)n-2(n-1)n + \theta_0+ \sum_{s\geq 3} \frac{(s-1)(s-2)}{2}|X_s|\\
	&=3n +\theta_0+ \sum_{s\geq 3} \frac{(s-1)(s-2)}{2}|X_s|.
	\end{align*}
	Plugging the first two equations in \eqref{eq:1:X_i_n=0} into the above one to replace $3n$, we obtain
	\begin{align*}
	\sum_{i=1}^{M_0}|X_i|&=\sum_{i\geq 0}|X_{3i+2}| + \ell_0 +2|\Delta_0^2|+ \sum_{i\geq 0}|X_{3i+1}|+\ell_0 - |\Delta_0^2|+\theta_0+ \sum_{s\geq 3} \frac{(s-1)(s-2)}{2}|X_s|,
	\end{align*}
	which implies
	\[ 0=2\ell_0 + |\Delta_0^2|+\theta_0+\sum_{s>3, 3\mid s} \left(\frac{(s-1)(s-2)}{2}-1\right)|X_s| + \sum_{s\geq 3, 3\nmid s} \left(\frac{(s-1)(s-2)}{2}\right)|X_s|. \]
	As $|X_s|$, $|\Delta_0^2|$ and $\ell_0$ are all nonnegative integers, the equation above holds only if $\ell_0=|\Delta_0^2|=\theta_0=0$, $|X_s|=0$ for $s>3$. By \eqref{eq:1:X_i_n=0},
	\[
	|X_1|= n+2, ~|X_2|=2n-2.
	\]
	%	and 
	%	\[|X_0|=(n^2+n+1)-|X_0|-|X_1|=(n-1)^2.\]
	Plugging them into \eqref{eq:sum_i|X_i|}, we get
	\begin{align*}
	|X_1|+2|X_2|+3|X_3|= 5n-2+3|X_3|=2n^2+n,
	\end{align*}
	which means $|X_3|=\frac{2}{3}(n-1)^2$.  On the other hand, \eqref{eq:sum_|X_i|} now becomes
	\[ |X_0|+|X_1|+|X_2|+|X_3| =n^2+n+1.\]
	This implies
	\[
	|X_0|=\frac{1}{3}(n-1)^2. 
	\]
	By $\ell_0=0$ and $|\Delta_0^2|=0$, we get $T_1\cap T_0^{(3)}=\emptyset$ and $\Delta_0\subseteq T_1$, respectively. Hence (b) is also proved.
	
	Finally, let us turn to (c). By 
	$\hat{T}^{(2)}T_1 =\sum_{i=1}^{M_1}i Y_i$
	and
	\[\hat{T}^{(2)}T_1 \equiv T_0-T_1^{(3)} + 2T_1\pmod{3},\]
	we have
	\begin{align*}
		\bigcup_{i\geq 0}Y_{3i+2} &= (T_1\setminus T_1^{(3)})~\dot{\cup}~ \left(T_1^{(3)}\setminus (T_0~\dot{\cup}~T_1 \cup\Delta_1)\right)~\dot{\cup}~ \Delta_{1}^0,\\
		\bigcup_{i\geq 0}Y_{3i+1} &= (T_0\setminus T_1^{(3)})~\dot{\cup}~\Delta_1^2~\dot{\cup}~(T_1\cap T_1^{(3)}\setminus \Delta_1^1).
	\end{align*}
	The above two equations implies
	\begin{align}
		\sum_{i\geq 0}|Y_{3i+2}|&=2n+2-2u_1-u_0-|\Delta_1^1|-2|\Delta_1^2|, \label{eq:Y_3i+2,n=1mod6}\\
		\sum_{i\geq 0}|Y_{3i+1}|&=n-u_0+u_1-|\Delta_1^1|+|\Delta_1^2|, \label{eq:Y_3i+1,n=1mod6}\,\\
		\sum_{i\geq 0}|Y_{3i}|&=n^2-2n-1+2u_0+u_1+2|\Delta_1^1|+|\Delta_1^2|,\label{eq:Y_3i,n=1mod6}
	\end{align}
	where the last equation comes from $\sum_{i\geq 0}|Y_i|=n^2+n+1$.
	
	By \eqref{eq:sum_|Y_i|-0} and \eqref{eq:sum_i|Y_i|}, respectively, we have
	\begin{align}
		\sum_{s\geq 3} \frac{s(s-3)}{2}|Y_s|&=|Y_1|+|Y_2|-n-1-\theta_1, \label{eq:Ys,n=1mod6}\\
		\sum_{i\geq 1}i|Y_{i}|&=2n^2+3n+1. \label{eq:iY_i,n=1mod6}
	\end{align}
	By \eqref{eq:Ys,n=1mod6}, we have 
	\begin{align}
		\sum_{i\geq 2}|Y_{3i}|\leq \sum_{i\geq 2}\frac{3i(3i-3)}{2}|Y_{3i}|\leq \sum_{s\geq 4}\frac{s(s-3)}{2}|Y_{s}|\leq |Y_1|+|Y_2|-n-1-\theta_1. \label{eq:s(s-3)/2*Ys}
	\end{align}
	The equation \eqref{eq:Y_3i+2,n=1mod6} and \eqref{eq:Y_3i+1,n=1mod6} implies
	\begin{align*}
		|Y_1|+|Y_2|-n-1-\theta_1\leq &~(2n+2-2u_1-u_0-|\Delta_1^1|-2|\Delta_1^2|)+\\
		&~(n-u_0+u_1-|\Delta_1^1|+|\Delta_1^2|)-n-1-\theta_1\\
		\leq &~2n+1-u_1-2u_0-2|\Delta_1^1|-|\Delta_1^2|-\theta_1.
	\end{align*}
	
	The above inequality together with \eqref{eq:s(s-3)/2*Ys} and \eqref{eq:Y_3i,n=1mod6} imply
	\begin{equation}\label{eq:Y0+Y3_lb,n=1mod6}
		|Y_0|+|Y_3|\geq n^2-4n-2+\theta_1.
	\end{equation}
	By \eqref{eq:Y_3i,n=1mod6}, one can directly get
	\begin{equation}\label{eq:Y0+Y3_ub,n=1mod6}
	|Y_0|+|Y_3|\leq n^2-2n-1+2u_0+u_1+2|\Delta_1^1|+|\Delta_1^2|.
	\end{equation}
	
	By $\eqref{eq:iY_i,n=1mod6}-\eqref{eq:Y_3i+1,n=1mod6}-2\times\eqref{eq:Y_3i+2,n=1mod6}$, we obtain 
	\begin{align}
	\label{eq:Y_3+Y_i>4,n=1mod6} 
	3|Y_3|+\sum_{i\geq 4}s_i |Y_i|=2n^2-2n-3+3u_0+3u_1+3|\Delta_1^1|+3|\Delta_1^2|,\, 
	\end{align}
	where 
	$s_i=\begin{cases}
	i-1, & i\equiv 1\pmod{3};\\
	i-2, & i\equiv 2\pmod{3};\\
	i, & i\equiv 0\pmod{3}.
	\end{cases}$
	
	This implies
	\begin{align*}
	|Y_3|\leq \frac{2n^2-2n-3}{3}+u_0+u_1+|\Delta_1^1|+|\Delta_1^2|,
	\end{align*}
	which is the upper bound in \eqref{eq:Y_3bd,n=1mod6}.
	
	Note that $s_i\leq \frac{i(i-3)}{2}$ for $i\geq 5$. Then 
	\begin{align}
	\sum_{i\geq 4}s_i |Y_i|\leq |Y_4|+\sum_{s\geq 4}\frac{s(s-3)}{2}|Y_s|. \label{eq:sum_i>4_si_Y_i}
	\end{align}
	Substituting  \eqref{eq:Y_3i+1,n=1mod6}, \eqref{eq:s(s-3)/2*Ys} and \eqref{eq:sum_i>4_si_Y_i} into \eqref{eq:Y_3+Y_i>4,n=1mod6} , we get 
	\begin{align*}
	|Y_3|\geq& \frac{2n^2-2n-3+3u_0+3u_1+3|\Delta_1^1|+3|\Delta_1^2|}{3} - \frac{|Y_4|+\sum_{s\geq 4}\frac{s(s-3)}{2}|Y_s|}{3}\\
	\geq&  \frac{2n^2-2n}{3} -1+u_0+u_1+|\Delta_1^1|+|\Delta_1^2|- \frac{n-u_0+u_1-|\Delta_1^1|+|\Delta_1^2|}{3}\\
			& -\frac{2n+1-u_1-2u_0-2|\Delta_1^1|-|\Delta_1^2|-\theta_1}{3}\\
	\geq& \frac{2n^2-5n-4+\theta_1}{3}+2u_0+u_1+2|\Delta_1^1|+ |\Delta_1^2|,
	\end{align*}
	which is the lower bound in \eqref{eq:Y_3bd,n=1mod6}.
	Moreover, by the \eqref{eq:Y0+Y3_lb,n=1mod6}, \eqref{eq:Y0+Y3_ub,n=1mod6} and \eqref{eq:Y_3bd,n=1mod6},
	\begin{equation*}
	\frac{n^2-10n-3}{3}-u_0-u_1-|\Delta_1^1|-|\Delta_1^2|+\theta_1\leq |Y_0|\leq  \frac{n^2-n+1-\theta_1}{3}. \qedhere
	\end{equation*}
\end{proof}

The above result will be used in Section \ref{sec:1mod6}. Next we turn to the case with $n\equiv 2\pmod{3}$, for which we can only prove several bounds on the value of $|X_i|$'s and $|Y_i|$'s.

\begin{theorem}\label{th:n=5_mod6_Xi_Yi}
	For $n\equiv 5 \pmod{6}$, if there exist inverse-closed subsets $T_0$ and $T_1\subseteq H$ with $e\in T_0$ and $k_0+k_1=2n+1$ satisfying \eqref{eq:T0T1} and \eqref{eq:T0^2+T1^2}, then 
	\begin{equation}\label{eq:X0_range}
		|X_0|\leq \frac{n-2\ell_0+1-\theta_0}{3}, \quad \sum_{i\geq 5}\frac{s_i}{3}|X_i|\leq \frac{n-2\ell_0+1-\theta_0}{3}, 
%		\quad n\geq 2\ell_0-1,
	\end{equation}
	\begin{align}
	\label{eq:X1_range} \frac{2n^2-4n+7\ell_0+3+2\theta_0}{3}\leq &|X_1|\leq \frac{2n^2-\ell_0+7-2\theta_0}{3},\\
	 \label{eq:X4_range} \frac{n^2-6n+7\ell_0-1+2\theta_0}{3}\leq &|X_4|\leq \frac{n^2-3n+\ell_0+2-\theta_0}{3},
	\end{align}
	where $\ell_0=|T_1\cap T_0^{(3)}|$, and
	\begin{align}
	|Y_0|\leq \frac{n-2u+2-\theta_1}{3}, ~\sum_{i\geq 5}\frac{s_i}{3}&|Y_i|\leq \frac{n-2u+2-\theta_1}{3}, \label{eq:Y0_range} \\
%	n\geq 2u-1, 
	\frac{2n^2-6n+7u-4+2\theta_1}{3}\leq &|Y_1|\leq \frac{2n^2-2n-u+4-2\theta_1}{3} \label{eq:Y1_range},\\
		\frac{n^2-4n+7u-7+2\theta_1}{3}\leq &|Y_4|\leq \frac{n^2-n+u-1-\theta_1}{3} \label{eq:Y4_range},
	\end{align}
	where $u=|T_0\cap T_1^{(3)}|+|T_1\cap T_1^{(3)}|$.
\end{theorem}
\begin{proof}
	As $n$ is odd, $k_0=n$ and $k_1=n+1$. Now \eqref{eq:T0^3left} and \eqref{eq:T1^3left} modulo $3$ becomes
	\begin{align}
	\label{eq:T_0^3_mod6=5} \hat{T}^{(2)}T_0 &\equiv H+T_1+T_0-T_0^{(3)} \pmod{3},\\
	\label{eq:T_1^3_mod6=5}	\hat{T}^{(2)}T_1 &\equiv H+T_0+T_1-T_1^{(3)} \pmod{3}.
	\end{align}
	
	Let us first concentrate on the value of $|X_i|$'s.
	Since $|H|=n^2+n+1$ is relatively prime to $3$, $T_0^{(3)}$ contains no repeating elements. By Lemma \ref{le:T_0T_1_nece_condi} (h), $T_0\cap T_0^{(3)}=\{e\}$. Therefore, by comparing the coefficients of each element in both sides of \eqref{eq:T_0^3_mod6=5}, we obtain
	\begin{align}
		\sum_{i\geq 0}|X_{3i+2}|&=2n-\ell_0, \label{eq:X_3i+2,n=2}\\
		\sum_{i\geq 0}|X_{3i}|&=n-\ell_0-1,\label{eq:X_3i,n=2}\\
		\sum_{i\geq 0}|X_{3i+1}|&=n^2-2n+2\ell_0+2. \label{eq:X_3i+1,n=2}
	\end{align}
	
	Now the equations  \eqref{eq:T^2T_0} and \eqref{eq:sum_|X_i|-0} implies
	\begin{align}
	\sum_{i\geq 1}i|X_{i}|&=2n^2+n, \label{eq:iX_i,n=2}\\
	-|X_1|-|X_2|+\sum_{s\geq 3} \frac{s(s-3)}{2}|X_s|&=-3n-\theta_0,\label{eq:Xs,n=2}
	\end{align}
	where we should recall that $\theta_0=|(T_0^2\backslash T_0^{(2)})\cap \hat{T}^{(4)}|$.
	
	By computing the equations $\eqref{eq:X_3i+1,n=2}\times 2-\eqref{eq:iX_i,n=2}+\eqref{eq:Xs,n=2}+(\eqref{eq:X_3i+2,n=2}+\eqref{eq:X_3i,n=2})\times3$, we have 
	\begin{equation}\label{eq:X0_range1}
		3|X_0|+\sum_{i\geq 5}s_i|X_i|=n-2\ell_0+1-\theta_0,\text{~where }
		s_i=\begin{cases}
		\frac{(i-1)(i-4)}{2}, & i\equiv 1\pmod{3};\\
		\frac{(i-2)(i-3)}{2}, & \text{otherwise}.
		\end{cases}
	\end{equation}
	In particular $s_5=3$.
	This implies \eqref{eq:X0_range}.

	Moreover,  by computing $\eqref{eq:iX_i,n=2}-\eqref{eq:X_3i+1,n=2}$, we get
	\begin{equation*}
	3|X_4|+2|X_2|+3|X_3|+\sum_{i\geq 5}(i-\Delta_i)|X_i|=n^2+3n-2\ell_0-2,
	\end{equation*}
	where $\Delta_i=1$ if $i\equiv1\mod 3$ and  $\Delta_i=0$ otherwise. One can easily check that $\frac{5}{3}s_i\geq i-\Delta_i$ for $i\geq 5$. Hence we get $(i-\Delta_i)|X_i|\leq \frac{5}{3}s_i|X_i|$ for $i\geq 5$.
	
	Plugging the equations \eqref{eq:X_3i+2,n=2}, \eqref{eq:X_3i,n=2} and \eqref{eq:X0_range} into the above one to replace $|X_2|$, $|X_3|$ and $|X_i|$ for $i\geq 5$, we obtain
	\begin{align}
	|X_4|&\geq \frac{(n^2+3n-2\ell_0-2)-2(2n-\ell_0)-3(n-\ell_0-1)-\frac{5}{3}(n-2\ell_0+1-\theta_0)}{3}\nonumber\\
	&\geq \frac{3n^2-17n+19\ell_0-2+5\theta_0}{9}\label{eq:X4_lbound}\\
	&\geq \frac{n^2-6n+7\ell_0-1+2\theta_0}{3}, \nonumber
	\end{align}
	where the last inequality comes from the fact that $n- 2\ell_0+1-\theta_0\geq 0$.

	Now we rewrite equation \eqref{eq:X_3i,n=2} as follows.
	\begin{align}
		\label{eq:X_3i,i>1,n=2}
		\sum_{i\geq 1}|X_{3i}|&=n-\ell_0-1-x_0,
	\end{align}
	where $x_0:=|X_0|$.
	
	Consider the equation obtained by computing $\eqref{eq:iX_i,n=2}-\eqref{eq:X_3i+2,n=2}\times 2-\eqref{eq:X_3i,i>1,n=2}\times 3$, we get
	\begin{align}
		|X_1|+4|X_4|+\sum_{i\geq 5}t_i|X_i|&=2n^2-6n+5\ell_0+3+3x_0 \nonumber\\
		&\leq 2n^2-5n+3\ell_0+4-\theta_0, \label{eq:X1+X4}
	\end{align}
	where $t_i\geq 3$ for $i\geq 5$. Then $\eqref{eq:X_3i+1,n=2}$ and $\eqref{eq:X1+X4}$ implies
	\begin{equation*}
		|X_4|\leq \frac{n^2-3n+\ell_0+2-\theta_0}{3},
	\end{equation*}
	which combined with inequality \eqref{eq:X4_lbound} leads to \eqref{eq:X4_range}.
	
	Plugging the \eqref{eq:X4_range} and the second inequality of \eqref{eq:X0_range} into \eqref{eq:X_3i+1,n=2}, we also get \eqref{eq:X1_range}  for $|X_1|$.
	
	The proof of \eqref{eq:Y0_range}, \eqref{eq:Y1_range} and \eqref{eq:Y4_range} are quite similar to the above proof for $|X_i|$'s. By \eqref{eq:T_1^3_mod6=5}, we can show that 
	\begin{align}
	\sum_{i\geq 0}|Y_{3i+2}|&=2n+1-u, \label{eq:Y_3i+2,n=2}\\
	\sum_{i\geq 0}|Y_{3i}|&=n+1-u,\label{eq:Y_3i,n=2}\\
	\sum_{i\geq 0}|Y_{3i+1}|&=n^2-2n-1+2u. \label{eq:Y_3i+1,n=2}
	\end{align}
	By the definition of $u$,  $u$ must be even and $u\leq n+1$.
	Combined with the condition $k_0=n,k_1=n+1$, the equations  \eqref{eq:T^2T_1} and \eqref{eq:sum_|Y_i|-0} implies
	\begin{align*}
		\sum_{i\geq 1}i|Y_{i}|&=2n^2+3n+1,\\
		-|Y_1|-|Y_2|+\sum_{s\geq 3} \frac{s(s-3)}{2}|Y_s|&=-n-1-\theta_1,
	\end{align*}
	where $\theta_1$ is defined in \eqref{eq:theta_1}. By the same calculation for $|X_i|$'s, we obtain \eqref{eq:Y0_range}, \eqref{eq:Y1_range} and \eqref{eq:Y4_range} and we omit the details.
\end{proof}

The proof for $n\equiv 2 \pmod{6}$ can be obtained in the same fashion as in the proof of Theorem \ref{th:n=5_mod6_Xi_Yi}. Hence we omit it.
\begin{theorem}\label{th:n=2_mod6_Xi_Yi}
	For $n\equiv 2 \pmod{6}$, if there exist inverse-closed subsets $T_0$ and $T_1\subseteq H$ with $e\in T_0$ and $k_0+k_1=2n+1$ satisfying \eqref{eq:T0T1} and \eqref{eq:T0^2+T1^2}, then 
	\begin{align}
		|X_0|\leq \frac{n-2\ell_0-\theta_0}{3}, ~\sum_{i\geq 5}\frac{s_i}{3}&|X_i|\leq \frac{n-2\ell_0-\theta_0}{3},  \\
		|Y_0|\leq \frac{n-2u+3-\theta_1}{3}, ~\sum_{i\geq 5}\frac{s_i}{3}&|Y_i|\leq \frac{n-2u+3-\theta_1}{3}, \\
		\frac{n^2-4n+7\ell_0+2\theta_0}{3}\leq &|X_4|\leq \frac{n^2-n+\ell_0-\theta_0}{3} \label{eq:X4_range_2n},\\
		\frac{2n^2-6n+7\ell_0+3+2\theta_0}{3}\leq &|X_1|\leq \frac{2n^2-2n-\ell_0+3-2\theta_0}{3} \label{eq:X1_range_2n},\\
		\frac{n^2-6n+7u-8+2\theta_1}{3}\leq &|Y_4|\leq \frac{n^2-3n+u+1-\theta_1}{3} \label{eq:Y4_range_2n},\\
		\frac{2n^2-4n+7u-4+2\theta_1}{3}\leq &|Y_1|\leq \frac{2n^2-u+8-2\theta_1}{3} \label{eq:Y1_range_2n}.
	\end{align}
\end{theorem}

\section{On $\hat{T}^{(4)}T_0$ and $\hat{T}^{(4)}T_1$}\label{sec:T^4T_i}
In this section,  we investigate $\hat{T}^{(4)}T_0$ and $\hat{T}^{(4)}T_1$. The following lemma tells us that $\hat{T}^{(4)}T_i\in \Z[H]$ is a linear combination of $H$, $T_0$, $T_1$, $X_j$'s, $Y_j$'s and $T_i^5$.
\begin{lemma}\label{le:T4Ti}
	Let $H$ be an abelian group of order $n^2+n+1$. If $H$ contains two inverse-closed subsets $T_0$ and $T_1$ such that the identity element $e\in T_0$, and \eqref{eq:T0T1} and \eqref{eq:T0^2+T1^2} holds, then
	\begin{align}
	\hat{T}^{(4)}T_0=(5k_0+(2k_0-k_1)(k_0^2-1))H+(4n^2+2n-3)T_0+2nT_1-4n\hat{T}^{(2)}T_0-\hat{T}^{(2)}T_1-T_0^5,\tag{\ref{eq:T0hatT^4}}\\
	\hat{T}^{(4)}T_1=(5k_1+(2k_1-k_0)(k_1^2-1))H+(4n^2+2n-3)T_1+2nT_0-4n\hat{T}^{(2)}T_1-\hat{T}^{(2)}T_0-T_1^5. \tag{\ref{eq:T1hatT^4}}
	\end{align}
\end{lemma}
\begin{proof}
	By multiplying $T_0^2$ (resp. $T_1^2$) on both side of $(\ref{eq:T0^3left})$ (resp. $(\ref{eq:T1^3left})$), we obtain
	\begin{align}
	T_0^5=(2k_0-k_1)k_0^2H+2nT_0^3+T_0^2T_1-T_0^3\hat{T}^{(2)}\, , \label{eq:T0^5}\\
	T_1^5=(2k_1-k_0)k_1^2H+2nT_1^3+T_0T_1^2-T_1^3\hat{T}^{(2)}\, .  \nonumber
%	\label{eq:T1^5}
	\end{align}
	By multiplying $T_0^{(2)}$ and $T_1^{(2)}$ on both sides of $(\ref{eq:T0^3left})$, 
	\begin{align}
	T_0^3T_0^{(2)}=(2k_0-k_1)k_0H+2nT_0T_0^{(2)}+T_1T_0^{(2)}-T_0(T_0^{(2)})^2-T_0T_1^{(2)}T_0^{(2)}\, , \label{eq:T0^3T0^{(2)}} \\
	T_0^3T_1^{(2)}=(2k_0-k_1)k_1H+2nT_0T_1^{(2)}+T_1T_1^{(2)}-T_0T_0^{(2)}T_1^{(2)}-T_0(T_1^{(2)})^2\, . \nonumber
	%\label{eq:T0^3T1^{(2)}}  
	\end{align}
	This implies
	\begin{equation*}
	T_0^3T_0^{(2)}+T_0^3T_1^{(2)}=(2k_0-k_1)(k_0+k_1)H+2nT_0\hat{T}^{(2)}+T_1\hat{T}^{(2)}-T_0((T_0^{(2)})^2+(T_1^{(2)})^2)-2T_0T_1^{(2)}T_0^{(2)}\, . 
	\end{equation*}
	By $(\ref{eq:T0T1})$ and $(\ref{eq:T0^2+T1^2})$, we have
	\begin{align*}
	T_0((T_0^{(2)})^2+(T_1^{(2)})^2)&=T_0(2H-T_0^{(2)}-T_1^{(2)}+2ne)^{(2)} \\
	&=T_0(2H-T_0^{(4)}-T_1^{(4)}+2ne)   \\
	&=2k_0H-T_0T_0^{(4)}-T_0T_1^{(4)}+2nT_0  ,
	\end{align*}
	and 
	\begin{equation*}
	2T_0T_1^{(2)}T_0^{(2)}=2T_0(H-e)^{(2)}=2k_0H-2T_0.
	\end{equation*}
	Plugging them all into \eqref{eq:T0^5}, we have
	\begin{equation*}
	T_0^5=(5k_0+(2k_0-k_1)(k_0^2-1))H+(4n^2+2n-3)T_0+2nT_1-4n\hat{T}^{(2)}T_0-\hat{T}^{(2)}T_1-\hat{T}^{(4)}T_0\, .
	\end{equation*}
	By a similar fashion, one can also derive that 
	\begin{equation*}
	T_1^5=(5k_1+(2k_1-k_0)(k_1^2-1))H+(4n^2+2n-3)T_1+2nT_0-4n\hat{T}^{(2)}T_1-\hat{T}^{(2)}T_0-\hat{T}^{(4)}T_1\, .
	\end{equation*}
	Rearranging the terms, we get \eqref{eq:T0hatT^4} and \eqref{eq:T1hatT^4}.
\end{proof}

Next, we concentrate the elements in $\hat{T}^{(4)}T_0$ and $\hat{T}^{(4)}T_1$. Write
\begin{equation}
	\hat{T}^{(4)}T_0=\sum_{i=0}^{N_0}iU_i, \tag{\ref{eq:Ui}}
\end{equation}
and
\begin{equation}
	\hat{T}^{(4)}T_1=\sum_{i=0}^{N_1}iV_i, \tag{\ref{eq:Vi}}
\end{equation}
where $U_i$ and $V_i$ are subsets of $H$ such that $U_i$ (resp. $V_i$) consists of the elements appearing in $\hat{T}^{(4)}T_0$ (resp. $\hat{T}^{(4)}T_1$) exactly $i$ times, $N_0=\max\{i: U_i\neq \emptyset\}$ and $N_1=\max\{i: V_i\neq \emptyset\}$.

Set $b_0=0,b_i\in T_0^{(4)}$ for $i=0,\dots,k_0-1$ and $b_i\in T_1^{(4)}$ for $i\in k_0,\dots,2n$. 
As in the proof of \eqref{eq:sum_|X_i|-0}  and \eqref{eq:sum_|Y_i|-0}  for $X_i$'s and $Y_i$'s in \cite{xu_wcc_2022}, we have the following result.

\begin{lemma}\label{le:in-ex_V_n=1mod6}
	Let $k_0=|T_0|$ and $k_1=|T_1|$. For $U_i$ and $V_i$ defined in \eqref{eq:Ui} and \eqref{eq:Vi}, 
	\begin{align}
	\label{eq:sum_|V_i|-0}  \sum_{i=1}^{N_1}|V_i|=&(2n+1)k_1-2(k_1-1)k_1 +\eta+ \sum_{s\geq 3} \frac{(s-1)(s-2)}{2}|V_s|,
	\end{align}
	where $	\eta$ is a nonnegative integer. When $n$ is odd, 
	\begin{equation}\label{eq:2V1+3V2+3V3+2V4,n=1mod6}
		2|V_1|+3|V_2|+3|V_3|+2|V_4|\geq 2n^2+4n+2.
	\end{equation}
\end{lemma}
\begin{proof}
	By the inclusion--exclusion principle, we count the distinct elements in $\hat{T}^{(4)}T_1$,
	\begin{equation}\label{eq:expand_T4T_1}
	|\hat{T}^{(4)}T_1| = \sum_{i=0}^{2n} |b_iT_1| - \sum_{i<j} |b_iT_1 \cap b_jT_1| + \sum_{r\geq 3} (-1)^{r-1} |b_{i_1}T_1\cap b_{i_2}T_1 \cap \cdots \cap b_{i_r}T_1|,
	\end{equation}
	where $i_1<i_2<\cdots <i_r$ cover all the possible values. By the definition of $V_i$'s, the left-hand side of \eqref{eq:expand_T4T_1} equals $\sum_{i=1}^{N_1}|V_i|$.
	
	Now we determine the value of the different sums in \eqref{eq:expand_T4T_1}. First, $\sum_{i=0}^{2n} |b_iT_1|=(2n+1)k_1$. 
	
	An element $b_ix=b_jy\in b_iT_k\cap b_jT_k$ for some $x,y\in T_k$ if and only if 
	$$
	b_ib_j^{-1}=x^{-1}y\in T_k^2.
	$$
	By \eqref{eq:T0^2+T1^2}, there exist at most two  possible choices of $(x,y)$. Moreover, if $(x,y)$ is such that $b_ix=b_jy$, then $b_iy^{-1}=b_jx^{-1}$ which also provides a solution $(y^{-1}, x^{-1})$. Thus $|b_i T_k\cap b_j T_k| =1$ if and only if $x=y^{-1}$, which is equivalent to $b_ib_j^{-1}\in T_k^{(2)}$. Then 	for  any $i<j$ and $k=0,1$, we have
	$$
	|b_i T_k\cap b_j T_k| \in \{0,1,2\}.
	$$ 
	This implies 
	$$
	\sum_{i<j} |b_iT_1 \cap b_jT_1|\leq 2\cdot |T_1|\cdot (|T_1|-1)=2(k_1-1)k_1.
	$$
	Moreover, we set $\eta=2(k_1-1)k_1-\sum_{i<j} |b_iT_1 \cap b_jT_1|$.
	
	Finally, suppose that $g\in b_{i_1}T_1\cap b_{i_2}T_1 \cap \cdots \cap b_{i_r}T_1$ with $r\geq 3$. It means that $g\in V_s$ for some $s\geq 3$. Then the contribution for $g$ in the sum $\sum_{r\geq 3} (-1)^{r-1} | b_{i_1}T_1\cap b_{i_2}T_1 \cap \cdots \cap b_{i_r}T_1|$ is
	\[ (-1)^{3-1} \binom{s}{3}+(-1)^{4-1} \binom{s}{4} +\cdots =\binom{s}{2} -\binom{s}{1} + \binom{s}{0} = \frac{(s-1)(s-2)}{2}.\]
	Therefore,
	\[\sum_{r\geq 3} (-1)^{r-1} |b_{i_1}T_1\cap b_{i_2}T_1 \cap \cdots \cap b_{i_r}T_1| = \sum_{s\geq 3} |V_s|\frac{(s-1)(s-2)}{2}.\]
	
	Plugging them all into \eqref{eq:expand_T4T_1}, we get \eqref{eq:sum_|V_i|-0}.
	
	By the definition of $V_i$, we have
	\begin{align*}
	\sum_{i=1}^{N_1}i|V_i|=|\hat{T}^{(4)}T_1|=(2n+1)(n+1)\, .
	\end{align*}
	For odd $n$, the above equation and \eqref{eq:sum_|V_i|-0} implies
	\begin{equation*}
	2|V_1|+3|V_2|+3|V_3|+2|V_4|=2n^2+4n+2+\sum_{s\geq 5}\frac{s(s-5)}{2}|V_s|+\eta\geq 2n^2+4n+2. \qedhere
	\end{equation*}
\end{proof}

\section{Case with $n\equiv 2\pmod{3}$}\label{sec:2mod3}
In this section, we consider the case where $n\equiv 2 \pmod{3}$. Depending on the parity of $n$, the proof are separated into $2$ cases. We provide every detail for the case with odd $n$. For ever $n$, as the proof is very similar, we only provide a sketch of it.

As we mentioned in Section \ref{sec:outline}, we investigate \eqref{eq:T0hatT^4} or \eqref{eq:T1hatT^4} modulo $5$. Hence, for odd $n$, our proof needs to be further separated into 5 different cases according to the value of $n$ modulo $5$, which are contained in Propositions \ref{prop:n=0mod5}--\ref{prop:n=1mod5}. By Lemma \ref{le:T_0T_1_nece_condi} (e), we have
\[k_0=n,k_1=n+1,\]
for odd $n$.
\begin{proposition}\label{prop:n=0mod5}
	For any positive odd integer $n$ satisfying $n\geq 19$, $n\equiv 0 \pmod{5}$, and $n\equiv 2\pmod{3}$, there is no inverse-closed subsets $T_0$ and $T_1\subseteq H$ with $e\in T_0$ satisfying \eqref{eq:T0T1} and \eqref{eq:T0^2+T1^2}.
\end{proposition}
\begin{proof}
	Since $k_0=n,k_1=n+1$, we have $k_1^2-1\equiv 0 \,~{\rm mod}\,~ 5$. Hence the equation \eqref{eq:T1hatT^4} implies
	\begin{align*}
	\hat{T}^{(4)}T_1=2T_1-\hat{T}^{(2)}T_0-T_1^{(5)}=2T_1-\sum_{i=1}^{M_0}iX_i-T_1^{(5)}\,~ ({\rm mod~}5)\, .
	\end{align*}
	Recall that  $\hat{T}^{(4)}T_1=\sum_{i=0}^{N_1}iV_i$, we get
	\begin{align}
	X_1\backslash(T_1\cup T_1^{(5)})\subseteq \bigcup_{i\geq 0} V_{5_i+4},\\
	X_4\backslash(T_1\cup T_1^{(5)})\subseteq \bigcup_{i\geq 0} V_{5_i+1}.
	\end{align}
	Note that the equation  $\hat{T}^{(4)}T_1=\sum_{i=0}^{N_1}iV_i$	 also implies
	\begin{align}
	|\hat{T}^{(4)}T_1|=(2n+1)(n+1)=\sum_{i=1}^{N_1}i|V_i| \label{eq:i|V_i|=(2n+1)(n+1)}\, .
	\end{align}
	Hence we obtain
	\begin{align*}
		~&(2n+1)(n+1)\\
	\geq~& \sum_{i\geq 1}4|V_{5i+4}|+\sum_{i\geq 1}|V_{5i+1}|\\
	\geq~& 4|X_1\backslash(T_1\cup T_1^{(5)})|+|X_4\backslash(T_1\cup T_1^{(5)})|\\
	\geq~& 4|X_1|+|X_4|-4|(T_1\cup T_1^{(5)})|\\
	\geq~& 4\cdot\frac{2n^2-4n+7\ell_0+3+2\theta_0}{3}+ \frac{n^2-6n+7\ell_0-1+2\theta_0}{3}-4(2n+2)\\
	\geq~& \frac{9n^2-46n+35\ell_0-13+10\theta_0}{3}\, ,
	\end{align*}
	where the third inequality comes from the fact that the number of elements in $T_1\cup T_1^{(5)}$ in the disjoint union of $X_1$ and $X_4$ is at most the size of $T_1\cup T_1^{(5)}$, and the fourth inequality is obtained via \eqref{eq:X1_range} and \eqref{eq:X4_range}.
	
	This means 
	\[
		3n^2-55n-16\leq 3n^2-55n+35\ell_0-16+10\theta_0\leq 0,
	\]
	which implies $n\leq 18$. 
%	However, according to \textcolor{red}{\cite{he_nonexistence_abelian_2021}}, this is impossible.
\end{proof}

For the other value of $n\mod{5}$, the computation becomes more complicated, because both $\hat{T}^{(2)}T_0$ and $\hat{T}^{(2)}T_1$ appear in  \eqref{eq:T0hatT^4} and  \eqref{eq:T1hatT^4}. Thus we need to bound  the sizes of $X_1\cap Y_1$, $X_4\cap Y_4$, and etc.

Define 
\[
\sigma = |X_1\cap Y_1|.
\]
Then we have 
\begin{align}
	|X_1|-\sigma-|Y_2|-|Y_3|-\sum_{s\geq 5}|Y_s|\leq&|X_1\cap Y_4|\leq |X_1|-\sigma, \label{eq:X1capY4range} \\
	|Y_1|-\sigma-|X_2|-|X_3|-\sum_{s\geq 5}|X_s|\leq&|Y_1\cap X_4|\leq |Y_1|-\sigma, \label{eq:Y1capX4range} \\
	|X_4|-|Y_1\cap X_4|-|Y_2|-|Y_3|-\sum_{s\geq 5}|Y_s|\leq&|X_4\cap Y_4|\leq |X_4|-|Y_1\cap X_4|.  \label{eq:X4capY4range}
\end{align}
Hence we can get the bounds of the above intersections by $\eqref{eq:X_3i+2,n=2}-\eqref{eq:X_3i+1,n=2}$, \eqref{eq:X0_range}, \eqref{eq:X4_range}, \eqref{eq:X1_range} and  $\eqref{eq:Y0_range}-\eqref{eq:Y1_range}$. Moreover by definition of $X_i$ and $Y_j$, we have
\begin{align*}
	\sum_{i\geq 1,i\neq 4}|Y_i|&=n^2+n+1-|Y_4|-|Y_0|,\\
	\sum_{i\geq 1,i\neq 4}|X_i|&=n^2+n+1-|X_4|-|X_0|.
\end{align*}
Then by \eqref{eq:X4_range} and \eqref{eq:Y4_range},
\begin{align}
	-\sum_{i\geq 1,i\neq 4}|X_i|&\geq -(n^2+n+1)+\frac{n^2-6n+7\ell_0-1+2\theta_0}{3}=\frac{-2n^2-9n+7\ell_0-4+2\theta_0}{3},  \label{eq:Xi_ine4}\\
	-\sum_{i\geq 1,i\neq 4}|Y_i|&\geq -(n^2+n+1)+\frac{n^2-4n+7u-7+2\theta_1}{3}=\frac{-2n^2-7n+7u-10+2\theta_1}{3}\,\label{eq:Yi_ine4}.
\end{align}
On the other hand by equation \eqref{eq:Y_3i+2,n=2}, \eqref{eq:Y_3i,n=2} and the upper bound of $\sum_{i\geq 5}\frac{s_i}{3}|Y_i|$ in \eqref{eq:Y0_range}, we get
\begin{align}
	-\sum_{i\geq 2,i\neq 4}|Y_i|\geq -(2n+1-u)-(n+1-u)-\frac{n-2u+2-\theta_1}{3}=\frac{-10n+8u-8+\theta_1}{3}\,\label{eq:Yi_ine_1_4}.
\end{align}
Similarly we also get
\begin{align}
	-\sum_{i\geq 2,i\neq 4}|X_i|&\geq =-(2n-\ell_0)-(n-\ell_0-1)-\frac{n-2\ell_0+1-\theta_0}{3}=\frac{-10n+8\ell_0+2+\theta_0}{3}\, .\label{eq:Xi_ine_1_4}
\end{align}

Now we are ready to consider the rest four cases. 

\begin{proposition}\label{prop:n=2mod5}
	For any positive odd integer $n$ satisfying $n\geq 114, n\equiv 2 \pmod{5}$, and $n\equiv 2 \pmod{3}$, there is no inverse-closed subsets $T_0$ and $T_1\subseteq H$ with $e\in T_0$ satisfying \eqref{eq:T0T1} and \eqref{eq:T0^2+T1^2}.
\end{proposition}
\begin{proof}
	By $k_0=n,k_1=n+1$, the  equation \eqref{eq:T1hatT^4} implies
	\begin{align*}
	\hat{T}^{(4)}T_1\equiv 2H+2T_1+4T_0+4\sum_{i=1}^{M_0}iX_i+2\sum_{i=1}^{M_1}iY_i-T_1^{(5)}\pmod{5}.
	\end{align*}
	Recalling  $\hat{T}^{(4)}T_1=\sum_{i=0}^{N_1}iV_i$, we have
	\begin{align*}
		(X_1\cap Y_4)\backslash(T_0\cup T_1\cup T_1^{(5)}) &\subseteq \bigcup_{i\geq 0} V_{5i+4},\\
		(X_1\cap Y_1)\backslash(T_0\cup T_1\cup T_1^{(5)})&\subseteq \bigcup_{i\geq 0} V_{5i+3},\\
		(X_4\cap Y_4)\backslash(T_0\cup T_1\cup T_1^{(5)})&\subseteq \bigcup_{i\geq 0} V_{5i+1}.
	\end{align*}
	
	Note that $\hat{T}^{(4)}T_1=\sum_{i=0}^{N_1}iV_i$ also implies
	\[|\hat{T}^{(4)}T_1|=(2n+1)(n+1)=\sum_{i=1}^{N_1}i|V_i| .\]
	Then we get
	\begin{align*}
		&(2n+1)(n+1)\\
	\geq &\sum_{i\geq 1}|V_{5i+1}|+\sum_{i\geq 1}3|V_{5i+3}|+\sum_{i\geq 1}4|V_{5i+4}|\\
	\geq &|X_4\cap Y_4|+3|X_1\cap Y_1|+4|X_1\cap Y_4|-4(|T_0|+|T_1|+|T_1^{(5)}|)\\
	\geq &(|X_4|-|Y_1|+\sigma-|Y_2|-|Y_3|-\sum_{s\geq 5}Y_s)+3\sigma\\
	&+4(|X_1|-\sigma-|Y_2|-|Y_3|-\sum_{s\geq 5}|Y_s|)-4(3n+2)\\
	= &|X_4|+4|X_1|-\sum_{i\geq 1,i\ne 4}|Y_s|-4\sum_{i\geq 2,i\ne 4}|Y_s|-4(3n+2)\\
	\geq &\frac{n^2-6n+7\ell_0-1+2\theta_0}{3}+4\cdot\frac{2n^2-4n+7\ell_0+3+2\theta_0}{3}+\frac{-2n^2-7n+7u-10+2\theta_1}{3}\\
	&+4\cdot\frac{-10n+8u-8+\theta_1}{3}-4(3n+2)\\
	\geq &\frac{7n^2-105n+39\ell_0+35u+9+10\theta_0+6\theta_1}{3},
	\end{align*}
	where the third inequality comes from \eqref{eq:Y1capX4range} and \eqref{eq:X4capY4range}, and the fourth inequality comes from \eqref{eq:Yi_ine_1_4}, \eqref{eq:Xi_ine_1_4}  together with the bounds for $|X_1|,|X_4|,|Y_1|$ in Theorem \ref{th:n=5_mod6_Xi_Yi}.
	
	This implies
	$$
	\frac{n^2-114n+39\ell_0+35u+6+10\theta_0+6\theta_1}{3}\leq 0.
	$$
	Hence
	\[
		n^2 -114n +6\leq 0.
	\]
	Then we must have  $n\leq 113$.
\end{proof}

\begin{proposition}\label{prop:n=3mod5}
	For any positive odd integer $n$ satisfying $n\geq 154, n\equiv 3 \pmod{5}$, and $n\equiv 2 \pmod{3}$, there is no inverse-closed subsets $T_0$ and $T_1\subseteq H$ with $e\in T_0$ satisfying \eqref{eq:T0T1} and \eqref{eq:T0^2+T1^2}.
\end{proposition}
\begin{proof}
	Now \eqref{eq:T1hatT^4} implies
	\begin{align*}
	\hat{T}^{(4)}T_1\equiv 4T_1+T_0+4\sum_{i=1}^{M_0}iX_i+3\sum_{i=1}^{M_1}iY_i-T_1^{(5)}\pmod{5}.
	\end{align*}
	Again we have 
	\begin{align*}
		(X_4\cap Y_1)\backslash(T_0\cup T_1\cup T_1^{(5)}) &\subseteq \bigcup_{i\geq 0} V_{5i+4},\\
		(X_4\cap Y_4)\backslash(T_0\cup T_1\cup T_1^{(5)})&\subseteq \bigcup_{i\geq 0} V_{5i+3},\\
		(X_1\cap Y_1)\backslash(T_0\cup T_1\cup T_1^{(5)})&\subseteq \bigcup_{i\geq 0} V_{5i+2},\\
		(X_1\cap Y_4)\backslash(T_0\cup T_1\cup T_1^{(5)})&\subseteq \bigcup_{i\geq 0} V_{5i+1}.
	\end{align*}
	By $|\hat{T}^{(4)}T_1|=(2n+1)(n+1)=\sum_{i=1}^{N_1}i|V_i|$, we get
	\begin{align*}
			&(2n+1)(n+1)\\
		\geq &\sum_{i\geq 1}|V_{5i+1}|+\sum_{i\geq 1}2|V_{5i+2}|+\sum_{i\geq 1}3|V_{5i+3}|+\sum_{i\geq 1}4|V_{5i+4}|\\
		\geq &|X_1\cap Y_4|+2|X_1\cap Y_1|+3|X_4\cap Y_4|+4|X_4\cap Y_1|-4(|T_0|+|T_1|+|T_1^{(5)}|)\\
		\geq &(|X_1|-\sigma-|Y_2|-|Y_3|-\sum_{s\geq 5}|Y_s|)+2\sigma+3(|X_4|-|Y_1|+\sigma-|Y_2|-|Y_3|-\sum_{s\geq 5}|Y_s|)\\
		&+4(|Y_1|-\sigma-|X_2|-|X_3|-\sum_{s\geq 5}|X_s|)-4(3n+2)\\
		= &|X_1|+3|X_4|+|Y_1|-4\sum_{i\geq 2,i\ne 4}|X_i|-4\sum_{i\geq 2,i\ne 4}|Y_i|-4(3n+2)\\
		\geq &\frac{2n^2-4n+7\ell_0+3+2\theta_0}{3}+3\cdot\frac{n^2-6n+7\ell_0-1+2\theta_0}{3} +\frac{2n^2-6n+7u-4+2\theta_1}{3}\\
		&+4\cdot\frac{-10n+8\ell_0+2+\theta_0}{3}+4\cdot\frac{-10n+8u-8+\theta_1}{3}-4(3n+2)\\
		\geq &\frac{7n^2-144n+60\ell_0+39u-52+12\theta_0+6\theta_1}{3},
	\end{align*}
	where the third inequality comes from \eqref{eq:X1capY4range}, \eqref{eq:Y1capX4range} and \eqref{eq:X4capY4range}, and the fourth inequality comes from \eqref{eq:Yi_ine_1_4}, \eqref{eq:Xi_ine_1_4}  together with the bounds for $|X_1|,|X_4|,|Y_1|$ in Theorem \ref{th:n=5_mod6_Xi_Yi}.
		
	This implies 
	$$
	\frac{n^2-153n+60\ell_0+39u-55+12\theta_0+6\theta_1}{3}\leq 0.
	$$
	Then we must have $n\leq 153$.
\end{proof}

The proof for $n\equiv 1, 4\pmod{5}$ becomes more complicated, because we have to consider both $\hat{T}^{(4)} T_0$ and $\hat{T}^{(4)} T_1$.
\begin{proposition}\label{prop:n=4mod5}
	For any positive odd integer $n$ satisfying $n\geq 232, n\equiv 4 \pmod{5}$ and $n\equiv 2 \pmod{3}$, there is no inverse-closed subsets $T_0$ and $T_1\subseteq H$ with $e\in T_0$ satisfying \eqref{eq:T0T1} and \eqref{eq:T0^2+T1^2}.
\end{proposition}
\begin{proof}
	Now $k_0^2-1=n^2-1\equiv 0 \,~{\rm mod}\,~ 5$. Hence the equation \eqref{eq:T0hatT^4} implies
	\begin{align*}
	\hat{T}^{(4)}T_0\equiv 4T_0+3T_1+4\sum_{i=1}^{M_0}iX_i+4\sum_{i=1}^{M_1}iY_i-T_0^{(5)}\pmod{5}.
	\end{align*}
	Recalling  $\hat{T}^{(4)}T_0=\sum_{i=0}^{N_0}iU_i$,  we have
	\begin{align*}
	(X_1\cap Y_1)\backslash(T_0\cup T_1\cup T_0^{(5)})&\subseteq \bigcup_{i\geq 0} U_{5i+3},\\
	(X_4\cap Y_4)\backslash(T_0\cup T_1\cup T_0^{(5)})&\subseteq \bigcup_{i\geq 0} U_{5i+2}.
	\end{align*}
	Note that $\hat{T}^{(4)}T_0=\sum_{i=0}^{N_0}iU_i$ also implies
	\[|\hat{T}^{(4)}T_0|=(2n+1)n=\sum_{i=1}^{N_0}i|U_i| .\]
	
	Hence we obtain
	\begin{align*}
		&(2n+1)n\\
	\geq &\sum_{i\geq 1}2|U_{5i+2}|+\sum_{i\geq 1}3|U_{5i+3}|\\
	\geq &2|X_4\cap Y_4|+3|X_1\cap Y_1|-3|T_0\cup T_1 \cup T_0^{(5)}|\\
	\geq &2(|X_4|-|Y_1|+\sigma-|Y_2|-|Y_3|-\sum_{s\geq 5}Y_s)+3\sigma-3(3n+1)\\
	\geq &2(|X_4|+\frac{-2n^2-7n+7u-10+2\theta_1}{3})+5\sigma-3(3n+1)\\
	\geq &2\cdot\frac{n^2-6n+7\ell_0-1+2\theta_0}{3}+2\cdot\frac{-2n^2-7n+7u-10+2\theta_1}{3}+5\sigma-3(3n+1),
	\end{align*}
	where the third inequality comes from  \eqref{eq:X4capY4range}, and the last two inequalities come from  \eqref{eq:Yi_ine4} and \eqref{eq:X4_range}, respectively. This means 
	\begin{align}
	\sigma\leq \frac{8n^2+56n-14\ell_0-14u+31-4\theta_0-4\theta_1}{15}. \label{eq:sigma_n=4mod5}
	\end{align}
	Furthermore the equation \eqref{eq:T1hatT^4} implies
	\[
	\hat{T}^{(4)}T_1\equiv 4H+4T_1+3T_0+4\sum_{i=1}^{M_0}iX_i+4\sum_{i=1}^{M_1}iY_i-T_1^{(5)}\pmod{5}.
	\]
	Recalling  $\hat{T}^{(4)}T_1=\sum_{i=0}^{N_1}iV_i$, we have
	\begin{align*}
	\left((X_1\cap Y_4)\bigcup (X_4\cap Y_1)\right)\backslash(T_0\cup T_1\cup T_1^{(5)}) &\subseteq \bigcup_{i\geq 0} V_{5i+4},\\
	(X_1\cap Y_1)\backslash(T_0\cup T_1\cup T_1^{(5)})&\subseteq \bigcup_{i\geq 0} V_{5i+2},\\
	(X_4\cap Y_4)\backslash(T_0\cup T_1\cup T_1^{(5)})&\subseteq \bigcup_{i\geq 0} V_{5i+1}.
	\end{align*}
	By calculation, we get
	\begin{align*}
		&(2n+1)(n+1)\\
	\geq &\sum_{i\geq 1}|V_{5i+1}|+\sum_{i\geq 1}2|V_{5i+2}|+\sum_{i\geq 1}4|V_{5i+4}|\\
	\geq &|X_4\cap Y_4|+2|X_1\cap Y_1|+4(|X_1\cap Y_4|+|X_4\cap Y_1|)-4(|T_0|+|T_1|+|T_1^{(5)}|)\\
	\geq &(|X_4|-|Y_1|+\sigma-|Y_2|-|Y_3|-\sum_{s\geq 5}|Y_s|)+2\sigma+4(|X_1|-\sigma-|Y_2|-|Y_3|-\sum_{s\geq 5}|Y_s|\\
	&+|Y_1|-\sigma-|X_2|-|X_3|-\sum_{s\geq 5}|X_s|)-4(3n+2)\\
	= &|X_4|+4|X_1|+3|Y_1|-5\sum_{i\geq 2,i\ne 4}|Y_s|-4\sum_{i\geq 2,i\ne 4}|X_s|-4(3n+2)-5\sigma\\
	\geq &\frac{n^2-6n+7\ell_0-1+2\theta_0}{3}+4\cdot\frac{2n^2-4n+7\ell_0+3+2\theta_0}{3}+3\cdot\frac{2n^2-6n+7u-4+2\theta_1}{3}\\
	&+5\cdot\frac{-10n+8u-8+\theta_1}{3}+4\cdot\frac{-10n+8\ell_0+2+\theta_0}{3}-4(3n+2)-5\sigma\\
	\geq &\frac{15n^2-166n+61u+67\ell_0-57+14\theta_0+11\theta_1}{3}-5\sigma,
	\end{align*}
	where the third inequality comes from \eqref{eq:X1capY4range}, \eqref{eq:Y1capX4range} and \eqref{eq:X4capY4range}, and the fourth inequality comes from \eqref{eq:Yi_ine_1_4}, \eqref{eq:Xi_ine_1_4}  together with the bounds for $|X_1|,|X_4|,|Y_1|$ in Theorem \ref{th:n=5_mod6_Xi_Yi}.
	
	This implies
	\begin{align*}
	\sigma\geq \frac{9n^2-175n+61u+67\ell_0-60+14\theta_0+11\theta_1}{15}.
	\end{align*}
	Then by \eqref{eq:sigma_n=4mod5}, we must have 
	$$
	\frac{9n^2-175n+61u+67\ell_0-60+14\theta_0+11\theta_1}{15}\leq \frac{8n^2+56n-14\ell_0-14u+31-4\theta_0-4\theta_1}{15}
	$$
	which means $n\leq 231$.
\end{proof}

\begin{proposition}\label{prop:n=1mod5}
	For any positive odd integer $n$ satisfying $n\geq 207, n\equiv 1 \pmod{5}$ and $n\equiv 2 \pmod{3}$, there is no inverse-closed subsets $T_0$ and $T_1\subseteq H$ with $e\in T_0$ satisfying \eqref{eq:T0T1} and \eqref{eq:T0^2+T1^2}.
\end{proposition}
\begin{proof}
	Similar to the proof of Proposition \ref{prop:n=4mod5},  now \eqref{eq:T1hatT^4} implies
	\begin{align*}
	\hat{T}^{(4)}T_1\equiv 4H+3T_1+2T_0+4\sum_{i=1}^{M_0}iX_i+\sum_{i=1}^{M_1}iY_i-T_1^{(5)}\pmod{5}.
	\end{align*}
	By $\hat{T}^{(4)}T_1=\sum_{i=0}^{N_1}iV_i$,
	\begin{align*}
	\left((X_1\cap Y_1)\bigcup (X_4\cap Y_4)\right)\backslash(T_0\cup T_1\cup T_1^{(5)}) &\subseteq \bigcup_{i\geq 0} V_{5i+4},\\
	(X_1\cap Y_4)\backslash(T_0\cup T_1\cup T_1^{(5)})&\subseteq \bigcup_{i\geq 0} V_{5i+2},\\
	(X_4\cap Y_1)\backslash(T_0\cup T_1\cup T_1^{(5)})&\subseteq \bigcup_{i\geq 0} V_{5i+1}.
	\end{align*}
	Furthermore,
	\begin{align*}
		&(2n+1)(n+1)\\
		\geq &\sum_{i\geq 1}|V_{5i+1}|+\sum_{i\geq 1}2|V_{5i+2}|+\sum_{i\geq 1}4|V_{5i+4}|\\
		\geq &|X_4\cap Y_1|+2|X_1\cap Y_4|+4|X_4\cap Y_4|+4|X_1\cap Y_1|-4(|T_0|+|T_1|+|T_1^{(5)}|)\\
		\geq &(|Y_1|-\sigma-|X_2|-|X_3|-\sum_{s\geq 5}|X_s|)+2(|X_1|-\sigma-|Y_2|-|Y_3|-\sum_{s\geq 5}|Y_s|)\\
		&+4(|X_4|-|Y_1|+\sigma-|Y_2|-|Y_3|-\sum_{s\geq 5}Y_s)+4\sigma-4(3n+2)\\
		= &2|X_1|+4|X_4|+|Y_1|+5\sigma-\sum_{i\geq 2,i\ne 4}|X_i|-4\sum_{i\geq 1,i\ne 4}|Y_i|-2\sum_{i\geq 2,i\ne 4}|Y_i|-4(3n+2)\\
		\geq &2\cdot\frac{2n^2-4n+7\ell_0+3+2\theta_0}{3}+4\cdot\frac{n^2-6n+7\ell_0-1+2\theta_0}{3} +\frac{2n^2-6n+7u-4+2\theta_1}{3}\\
		&+\frac{-10n+8\ell_0+2+\theta_0}{3}+4\cdot\frac{-2n^2-7n+7u-10+2\theta_1}{3}+2\cdot\frac{-10n+8u-8+\theta_1}{3}\\
		&-4(3n+2)+5\sigma\\
		\geq &\frac{2n^2-132n+50\ell_0+51u-80+13\theta_0+12\theta_1}{3}+5\sigma.
	\end{align*}
	It implies 
	$$
	\sigma\leq \frac{4n^2+141n-50\ell_0-51u+83-13\theta_0-12\theta_1}{15}.
	$$
	On the other hand, we have 
	\begin{align*}
	|X_1\cup Y_1|=|X_1|+|Y_1|-|X_1\cap Y_1|\leq |H|=n^2+n+1,
	\end{align*}
	which implies
	\begin{align*}
	\sigma&\geq\frac{2n^2-4n+7\ell_0+3+2\theta_0}{3}+\frac{2n^2-6n+7u-4+2\theta_1}{3}-(n^2+n+1)\\
	&=\frac{n^2-13n+7u+7\ell_0-4+2\theta_0+2\theta_1}{3}.
	\end{align*}
	Then  we must have 
	$$
	\frac{n^2-13n+7u+7\ell_0-4+2\theta_0+2\theta_1}{3}\leq \frac{4n^2+141n-50\ell_0-51u+83-13\theta_0-12\theta_1}{15}.
	$$	
	Consequently,
	\[
	 n^2 -206n +58u+57\ell_0+15\theta_0+14\theta_1 -103\leq 0,
	\]
	which implies $n\leq 206$.
\end{proof}

The following theorem is a summary of Propositions \ref{prop:n=0mod5}, \ref{prop:n=2mod5}, \ref{prop:n=3mod5}, \ref{prop:n=4mod5} and \ref{prop:n=1mod5}.
\begin{theorem}\label{th:n=2mod3_odd}
	Let  $n$ be a positive odd number with $n\equiv 2\pmod{3}$  satisfying one of the following conditions. 
	\begin{enumerate}[label=(\alph*)]
		\item  $n\geq 19, n\equiv 0\pmod{5}$;
		\item  $n\geq 207, n\equiv 1\pmod{5}$;
		\item  $n\geq 114, n\equiv 2\pmod{5}$;
		\item  $n\geq 154, n\equiv 3\pmod{5}$;
		\item  $n\geq 232, n\equiv 4\pmod{5}$.
	\end{enumerate}
	Then there is no inverse-closed subsets $T_0$ and $T_1\subseteq H$ with $e\in T_0$ satisfying \eqref{eq:T0T1} and \eqref{eq:T0^2+T1^2}.
\end{theorem}

By comparing  Theorem \ref{th:n=5_mod6_Xi_Yi} and Theorem \ref{th:n=2_mod6_Xi_Yi},  we see the bounds of $|X_i|$'s  and $|Y_i|$'s for $n\equiv 5\pmod{6}$ and $n\equiv 2\pmod{6}$ are very similar: the coefficients of $n^2$ in the bounds for the same $|X_i|$ (or the same $|Y_i|$) are always the same. It follows that the proof of the following result for even  $n$ with $n\equiv 2 \pmod{3}$ is more or less the same  as the proof of Theorem \ref{th:n=2mod3_odd}, and we omit it.
\begin{theorem}\label{th:n=2mod3_even}
	Let  $n$ be a positive even number with $n\equiv 2\pmod{3}$  satisfying one of the following conditions. 
	\begin{enumerate}[label=(\alph*)]
		\item  $n\geq 20, n\equiv 0\pmod{5}$;
		\item  $n\geq 206, n\equiv 1\pmod{5}$;
		\item  $n\geq 115, n\equiv 2\pmod{5}$;
		\item  $n\geq 163, n\equiv 3\pmod{5}$;
		\item  $n\geq 235, n\equiv 4\pmod{5}$.
	\end{enumerate}
	Then there is no inverse-closed subsets $T_0$ and $T_1\subseteq H$ with $e\in T_0$ satisfying \eqref{eq:T0T1} and \eqref{eq:T0^2+T1^2}.
\end{theorem}

\section{Case with $n\equiv 1\pmod{6}$}\label{sec:1mod6}
In this part, we consider the last case of $n$ modulo $6$.
\begin{proposition}\label{prop:n=0,3mod5_1mod6}
	For any positive integer $n$ satisfying $n\equiv 0,3\pmod{5}$ and $n\equiv 1\pmod{6}$, there is no inverse-closed subsets $T_0$ and $T_1\subseteq H$ with $e\in T_0$ satisfying \eqref{eq:T0T1} and \eqref{eq:T0^2+T1^2}.
\end{proposition}
\begin{proof}
	Since
	\[
	8n-7 \equiv \begin{cases}
	3\pmod 5, & n\equiv 0 \pmod{5},\\
	2\pmod 5, & n\equiv 3 \pmod{5},
	\end{cases}
	\]
	the integer $8n-7$ cannot be a square in $\Z$. Moreover, $3\mid n^2+n+1$. The nonexistence result follows directly from Corollary 3.1 in \cite{he_nonexistence_abelian_2021} which will be recalled in Proposition \ref{prop:known}.
\end{proof}
\begin{proposition}\label{prop:n=4mod5_1mod6}
	For any positive integer $n$ satisfying $n\geq \myue{83}, n\equiv 4\pmod{5}$ and $n\equiv 1 \pmod{6}$, there is no inverse-closed subsets $T_0$ and $T_1\subseteq H$ with $e\in T_0$ satisfying \eqref{eq:T0T1} and \eqref{eq:T0^2+T1^2}.
\end{proposition}
\begin{proof}
	The equation \eqref{eq:T0hatT^4} becomes
	\begin{align}\label{eq:T0hatT^4_4mod5_1mod6}
		\hat{T}^{(4)}T_0\equiv 4T_0+3T_1+4\sum_{i\geq 1}iX_i+4\sum_{i\geq 1}iY_i+4T_0^{(5)}\, ~({\rm mod~}5).
	\end{align}
	Put $R=4T_0+3T_1+4X_1+3X_2+4\sum_{i\geq 1,i\neq 3}iY_i+4T_0^{(5)}$. Denote by  $R_s={\rm Supp}(R)$ the support of $R$, that is 
	\[R_s =\{h\in H: \text{ the coefficient of $h$ in $R$ is positive}\}.\] 
	Note that by Theorem \ref{th:n=1_mod6_Xi} (a), we have $|X_i|=0$ for $i>3$.	Then \eqref{eq:T0hatT^4_4mod5_1mod6} equals
	\begin{align*}
	\hat{T}^{(4)}T_0\equiv  2Y_3+2X_3+R\, ~({\rm mod~}5).
	\end{align*}
	Furthermore
	\begin{align*}
		(X_3\cap Y_3)\backslash R_s &\subseteq \bigcup_{i\geq 0} U_{5i+4},\\
		(X_3\cup Y_3)\backslash((X_3\cap Y_3)\cup R_s)&\subseteq \bigcup_{i\geq 0} U_{5i+2}.
	\end{align*}
	Note that we have
	\begin{align}
	|R_s|\leq |T_0|+|T_1|+|X_1|+|X_2|+|Y_1|+|Y_2|+\sum_{i\geq 4}|Y_i|+|T_0^{(5)}|.
	\end{align}
	By Theorem \ref{th:n=1_mod6_Xi} (a), \eqref{eq:Y_3bd,n=1mod6} and \eqref{eq:Y_0bd,n=1mod6}, we get
	\begin{align*}
	|R_s|\leq &n+(n+1)+(n+2)+(2n-2)+|H|-(|Y_0|+|Y_3|)+n\\
	\leq &6n+1+(n^2+n+1)-|Y_0|-|Y_3|\\
	\leq &n^2+7n+2-\left(\frac{2n^2-5n-4+\theta_1}{3}+2u_0+u_1+2|\Delta_1^1|+|\Delta_1^2|\right)\\
	&-\left(\frac{n^2-10n-3}{3}-u_0-u_1-|\Delta_1^1|-|\Delta_1^2|+\theta_1\right)\\
	=  &\frac{36n+13-4\theta_1}{3}-u_0-|\Delta_1^1|.
	\end{align*}
	Then we get
	\begin{align*}
		|\hat{T}^{(4)}T_0|=&(2n+1)n\\
		\geq &\sum_{i\geq 1}2|U_{5i+2}|+\sum_{i\geq 1}4|U_{5i+4}|\\
		\geq &2(|X_3|+|Y_3|-2|X_3\cap Y_3|)+4|X_3\cap Y_3|-4|R_s|\\
		\geq &2\cdot\frac{2(n-1)^2}{3}+
		2\cdot \left(\frac{2n^2-5n-4+\theta_1}{3}+2u_0+u_1+2|\Delta_1^1|+|\Delta_1^2|\right)\\
		&-4(\frac{36n+13-4\theta_1}{3}-u_0-|\Delta_1^1|),
	\end{align*}
	which implies
	\begin{align*}
	\frac{2n^2-165n-56+18\theta_1}{3}+2(4u_0+u_1+4|\Delta_1^1|+|\Delta_1^2|)\leq 0.
	\end{align*}
	Since $\theta_1, u_0,u_1,|\Delta_1^1|,|\Delta_1^2|\geq 0$, we have $n\leq 82$.
\end{proof}

\begin{proposition}\label{prop:n=1mod5_1mod6}
	For any positive integer $n$ satisfying $n\geq 186, n\equiv 1\pmod{5}$ and $n\equiv 1 \pmod{6}$, there is no inverse-closed subsets $T_0$ and $T_1\subseteq H$ with $e\in T_0$ satisfying \eqref{eq:T0T1} and \eqref{eq:T0^2+T1^2}.
\end{proposition}
\begin{proof}
	The equation \eqref{eq:T1hatT^4} becomes
	\begin{align*}
	\hat{T}^{(4)}T_1\equiv 4H+ 3T_1+2T_0+4\sum_{i\geq 1}iX_i+\sum_{i\geq 1}iY_i-T_1^{(5)}\, ~({\rm mod~}5).
	\end{align*}
	Similar to the proof of Proposition \ref{prop:n=4mod5_1mod6}, we put $R=3T_1+2T_0+4X_1+3X_2+\sum_{i\geq 1,i\neq 3}iY_i+4T_1^{(5)}$. Denote by  $R_s={\rm Supp}(R)$ the support of $R$. 	Then the above equation becomes
	\begin{align*}
	\hat{T}^{(4)}T_0\equiv 4H+2X_3+ 3Y_3+R\, ~({\rm mod~}5).
	\end{align*}
	Again we have 
	\begin{align*}
		(H\backslash (X_3\cup Y_3)\cup (X_3\cap Y_3))\backslash R_s &\subseteq \bigcup_{i\geq 0} V_{5i+4},\\
		Y_3\backslash((X_3\cap Y_3)\cup R_s)&\subseteq \bigcup_{i\geq 0} V_{5i+2},\\
		X_3\backslash((X_3\cap Y_3)\cup R_s)&\subseteq \bigcup_{i\geq 0} V_{5i+1} .
	\end{align*}
	
	This means
	\begin{align}
		\nonumber	&(2n+1)(n+1)\\
		\nonumber	\geq &\sum_{i\geq 1}|V_{5i+1}|+\sum_{i\geq 1}2|V_{5i+2}|+\sum_{i\geq 1}4|V_{5i+4}|\\
		\nonumber	\geq & 4(|H|-|X_3|-|Y_3|+|X_3\cap Y_3|+|X_3\cap Y_3|)+2(|Y_3|-|X_3\cap Y_3|)\\
		\nonumber	&+(|X_3|-|X_3\cap Y_3|)-4|R_s|\\
		\label{eq:eps_bound_1_1mod5_1mod6}	= &4(n^2+n+1)-3|X_3|-2|Y_3|+5\epsilon-4|R_s|\\
		\nonumber	\geq & 4(n^2+n+1)-3\cdot\frac{2(n-1)^2}{3}+5\epsilon-4|R_s|\\
		\nonumber	&-2\cdot \left(\frac{2n^2-2n-3}{3}+u_0+u_1+|\Delta_1^1|+|\Delta_1^2|\right) ,
	\end{align}
	where $\epsilon :=|X_3\cap Y_3|$.
	
	Similarly we cat get an upper bound of $|R_s|$ 
	\begin{align*}
	|R_s|\leq &|T_1|+|T_0|+|X_1|+|X_2|+|Y_1|+|Y_2|+\sum_{i\geq 1,i\neq 3}|Y_i|+|T_1^{(5)}|\\
	\leq & (n+1)+n+(n+2)+(2n-2)+|H|-|Y_0|-|Y_3|+(n+1)\\
	\leq & 6n+2+(n^2+n+1)-|Y_0|-|Y_3|\\
	\leq & n^2+7n+3- \left(\frac{2n^2-5n-4+\theta_1}{3}+2u_0+u_1+2|\Delta_1^1|+|\Delta_1^2|\right)\\
	&-\left(\frac{n^2-10n-3}{3}-u_0-u_1-|\Delta_1^1|-|\Delta_1^2|+\theta_1\right)\\
	=  &\frac{36n+16-4\theta_1}{3}-u_0-|\Delta_1^1|.
	\end{align*}
	The above bound and \eqref{eq:eps_bound_1_1mod5_1mod6} imply
	\begin{align}\label{eq:eps_bound_2_1mod5_1mod6}
	\epsilon \leq \frac{4n^2+125n+55-16\theta_1}{15}+\frac{-2u_0+2u_1-2|\Delta_1^1|+2|\Delta_1^2|}{5}.
	\end{align}
	On the other hand, we have 
	\begin{align*}
	|X_3\cup Y_3|=|X_3|+|Y_3|-|X_3\cap Y_3|\leq |H|=n^2+n+1.
	\end{align*}
	By Theorem \ref{th:n=1_mod6_Xi} (a) and (c), this implies
	\begin{align*}
		\epsilon=&|X_3\cap Y_3|\\
		\geq &-(n^2+n+1)+|X_3|+|Y_3|\\ 
		\geq &-(n^2+n+1)+\frac{2(n-1)^2}{3}+
		\left(\frac{2n^2-5n-4+\theta_1}{3}+2u_0+u_1+2|\Delta_1^1|+|\Delta_1^2|\right)\\
		\geq & \frac{n^2-12n-5+\theta_1}{3}+2u_0+u_1+2|\Delta_1^1|+|\Delta_1^2|.
	\end{align*}
	Comparing the above bound and \eqref{eq:eps_bound_2_1mod5_1mod6}, we have
	\begin{align*}
		&\frac{n^2-12n-5+\theta_1}{3}+ {2u_0+u_1+2|\Delta_1^1|+|\Delta_1^2|}\\ \leq&\frac{4n^2+125n+55-16\theta_1}{15}+\frac{-2u_0+2u_1-2|\Delta_1^1|+2|\Delta_1^2|}{5},
	\end{align*}
	which means 
	\[
	\frac{n^2-185n-80}{15}+\frac{8u_0 +3u_1 +12|\Delta_1^1|+3|\Delta_1^2|}{5} \leq 0.
	\]
	Therefore $n\leq 185$.
\end{proof}

\begin{proposition}\label{prop:n=2mod5_1mod6}
	For any positive integer $n$ satisfying $n\geq 125, n\equiv 2\pmod{5}$ and $n\equiv 1 \pmod{6}$, there is no inverse-closed subsets $T_0$ and $T_1\subseteq H$ with $e\in T_0$ satisfying \eqref{eq:T0T1} and \eqref{eq:T0^2+T1^2}.
\end{proposition}
\begin{proof}
	The equation \eqref{eq:T1hatT^4} becomes
	\begin{align*}
	\hat{T}^{(4)}T_1\equiv 2H+ 2T_1+4T_0+4\sum_{i\geq 1}iX_i+2\sum_{i\geq 1}iY_i-T_1^{(5)}\, ~({\rm mod~}5).
	\end{align*}
	Put $R=2T_1+4T_0+4X_1+3X_2+4T_1^{(5)}+2\sum_{i\geq 1,i\neq 3}iY_i$. Define  $R_s={\rm Supp}(R)$. 
	Then the equation above becomes
	\begin{align*}
	\hat{T}^{(4)}T_1\equiv 2H+ Y_3+2X_3+R\, ~({\rm mod~}5).
	\end{align*}
	This implies
	\begin{align*}
	V_1\backslash R_s&\subseteq \emptyset\\
	V_2\backslash R_s&\subseteq H\backslash (X_3\cup Y_3)\\
	V_3\backslash R_s&\subseteq Y_3\backslash(X_3\cap Y_3)\\
	V_4\backslash R_s&\subseteq X_3\backslash(X_3\cap Y_3).
	\end{align*}
	Then we have
	\begin{align}
	2|V_1|+3|V_2|+3|V_3|+2|V_4|\leq & 3|R_s|+3(|H|-|X_3|-|Y_3|+|X_3\cap Y_3|) \nonumber\\
	&+3(|Y_3|-|X_3\cap Y_3|)+2(|X_3|-|X_3\cap Y_3|)\nonumber\\
	=&3|H|+3|R_s|-|X_3|-2\epsilon,\label{eq:2V1+3V2,ub,n=1mod6}
	\end{align}
	where $\epsilon :=|X_3\cap Y_3|$.
	
	By the definition of $R_s$, we have
	\begin{align*}
	|R_s|\leq &|T_0|+|T_1|+|Y_1|+|Y_2|+\sum_{i\geq 4}|Y_i|+|X_1|+|X_2|+|T_1^{(5)}|.
	\end{align*}
	By substituting the value of$|X_i|$ in Theorem \ref{th:n=1_mod6_Xi} (a) and the inequality \eqref{eq:Y0+Y3_lb,n=1mod6}, we get  
	\begin{align}
	\nonumber	|R_s| &= n+(n+1) + (n^2+n+1-|Y_0|-Y_3|) + (n+2)+(2n-2)+(n+1)\\
		&\leq 11n+5-\theta_1.\label{eq:Rs,n=1mod6}
	\end{align}
	On the other hand, by \eqref{eq:2V1+3V2+3V3+2V4,n=1mod6} and \eqref{eq:2V1+3V2,ub,n=1mod6} we have 
	\begin{align*}
	2n^2+4n+2\leq 2|V_1|+3|V_2|+3|V_3|+2|V_4|\leq 3|H|+3|R_s|-|X_3|-2\epsilon.
	\end{align*}
	Plugging  the bounds of $|R_s|$ above and $|X_3|$ into the inequality above, 
	\begin{align*}
		\epsilon \leq \frac{n^2+100n+46-9\theta_1}{6}.
	\end{align*} 
	Similar to the proof of case $ n\equiv 1 \pmod{5}$, we get 
	\begin{align*}
	\epsilon \geq  \frac{n^2-12n-5+\theta_1}{3}+u_0+|\Delta_1^1|.
	\end{align*}
	Then it must be true that   
	$$
	\frac{n^2-12n-5+\theta_1}{3}+u_0+|\Delta_1^1|\leq \frac{n^2+100n+46-9\theta_1}{6},
	$$
	which means $n\leq 124$.
\end{proof}

The following theorem is a summary of Propositions \ref{prop:n=0,3mod5_1mod6},  \ref{prop:n=4mod5_1mod6}, \ref{prop:n=1mod5_1mod6} and \ref{prop:n=2mod5_1mod6}.
\begin{theorem}\label{th:n=1mod6}
	Let  $n$ be a positive number with $n\equiv 1\pmod{6}$  satisfying one of the following conditions. 
	\begin{enumerate}[label=(\alph*)]
		\item  $n\equiv 0\pmod{5}$;
		\item  $n\geq 186, n\equiv 1\pmod{5}$;
		\item  $n\geq 125, n\equiv 2\pmod{5}$;
		\item  $n\equiv 3\pmod{5}$;
		\item  $n\geq 83, n\equiv 4\pmod{5}$.
	\end{enumerate}
	Then there is no inverse-closed subsets $T_0$ and $T_1\subseteq H$ with $e\in T_0$ satisfying \eqref{eq:T0T1} and \eqref{eq:T0^2+T1^2}.
\end{theorem}

\section{On the small value of $n$}\label{sec:small_n}
In \cite{he_thesis_2021,he_nonexistence_abelian_2021}, several different approaches have been applied to derive various nonexistence results of APLL codes of packing radius $2$. In particular, we need the following one.
\begin{proposition}[Corollary 3.1 in \cite{he_nonexistence_abelian_2021}]
	\label{prop:known}
	Suppose that $8n-7$ is not a square in $\Z$ and one of the following conditions are satisfied:
	\begin{itemize}
		\item $3,7,19$ or $31$ divides $n^2+n+1$;
		\item $13 \mid n^2+n+1$ and $8n-11\notin \{13k^2 : k\in\Z \}$.
	\end{itemize}
	Then there is no  APLL code of packing radius $2$ in $\Z^n$.
\end{proposition}

The symmetric polynomial method applied by Kim \cite{kim_2017_nonexistence} on the existence of perfect linear Lee codes can also be used directly here to derive strong nonexistence results; see Theorem 3.1 in \cite{he_nonexistence_abelian_2021}. 
%In particular, it leads to the following results appeared in \cite{he_nonexistence_abelian_2021}.
%	\begin{result}\label{result:prime}
%		For $3<n\leq 10^5$, if $n^2+n+1$ is a prime, then  there is no APLL code of packing radius $2$ in $\Z^n$.
%\end{result}
This approach can also be extended to the case in which $n^2+n+1$ contains a prime divisor $v>2n+1$, which can be found in \cite{zhang_nonexistence_2019} for perfect linear Lee codes and in \cite{he_thesis_2021} for almost perfect cases. 
\begin{proposition}[Theorem 4.6 in \cite{he_thesis_2021}]
	\label{prop:known_2}
	For $n>1$, suppose that $2(n^2+n+1)=mv$ where $v$ is prime and $v>2n+1$. Set
	\[
	a= \begin{cases}
		\min\{a\in \Z^+: v\mid4^a+4n+2 \}, & \text{if }\{a\in \Z^+: v\mid4^a+4n+2 \}\neq \emptyset;\\
		\infty, &\text{otherwise},
	\end{cases}
	\]
	and $b=\mathrm{ord}_v(4)$. Assume that there is an APLL code of packing radius $2$. Then there exists at least one $\ell\in \{0,1\cdots, \lfloor\frac{m}{4}\rfloor \}$ such that the equation
	\[
	a(x+1)+by=n-\ell
	\]
	has nonnegative integer solutions for $x$ and $y$.
\end{proposition}

\begin{proof}[Proof of Theorem \ref{th:main}]
	For $n>2$,  by Theorems \ref{th:n=2mod3_odd}, \ref{th:n=2mod3_even}, \ref{th:n=1mod6} and Proposition \ref{prop:known}, we only have to consider the following $23$ choices of $n$:
	\begin{align}
		\nonumber	\{&7, 8, 11, 14, 17, 29, 37, 38, 41, 47, 56, 59, 62, 67, 71, 77, 79, 89, \\
		\label{eq:n_s}	&92, 101, 104, 119, 121, 131, 134, 143, 161, 164, 176, 191, 194, 209\}.
	\end{align}

	It is easy to make a computer program to verify the necessary condition given in Proposition \ref{prop:known_2}. Consequently, we can exclude the numbers in \eqref{eq:n_s} except for the following $10$  values
	\begin{equation}\label{eq:10_n}
		n=
		11, 29, 47, 56, 67, 79, 104, 121, 134, 191.	
	\end{equation}
	Together with the nonexistence results for $n\equiv 0, 3,4 \pmod{6}$ in Theorem \ref{th:main_xu}, we get Theorem \ref{th:main}.	
\end{proof}

\section{Concluding remarks}\label{sec:concluding}
In this paper, we have obtained a strong necessary condition for the existence of linear Lee codes of packing radius $2$ and packing density $\frac{|S(n,2)|}{|S(n,2)|+1}$ in $\Z^n$. By Theorem \ref{th:main}, $n$ can only take $12$ different value. We conjecture that $n$ must be $1$ or $2$, which means the APLL code $\{6a :a\in \Z\}$ in $\Z$ and the APLL code in $\Z^2$ depicted in Figure \ref{fig:Z_14} are the only two examples up to isometry.

To prove this conjecture, we have to exclude the other $10$ possible value of $n$ in \eqref{eq:10_n}. The smallest number in this list is $n=11$, for which the size of $H$ is $133=7\cdot 19$. However, it seems not easy to conduct an exhaustive search for the existence of $T_0$ and $T_1$ satisfying Lemma \ref{le:group_ring_inter}: by the inverse-closed property of $T_i$ and $e\in T_0$, one has to check $\binom{66}{11}\approx 2^{40}$ possible cases of $(T_0,T_1)$. 

It is well known that multipliers for difference sets are very helpful for getting either the construction of a difference set or a nonexistence proof; see \cite[Chapter 6]{beth_design_1999}. One may also wonder whether there exist  multipliers for $T_0$ and $T_1$. Unfortunately, we could not find any multipliers except for $-1$.

\section*{Acknowledgment}
Zijiang Zhou is partially supported by the National Natural Science Foundation of China (No.~62202475) and the Natural Science Foundation of Hunan Province (No.~2021JJ40701). Yue Zhou is partially supported by the Training Program for Excellent Young  Innovators of Changsha (No.~kq2106006) and the Fund for NUDT Young Innovator Awards (No.~20180101).

%\bibliographystyle{abbrv}
%\bibliography{C:/Documents/References/Reference_math}
\end{document}